 \documentclass[reqno]{amsart}
 \usepackage[latin9]{inputenc}
 \usepackage{float}
 \usepackage{textcomp}
 \usepackage{mathrsfs}
 \usepackage{mathtools}
 \usepackage{amstext}
 \usepackage{amsthm}
 \usepackage{amssymb}
 \usepackage{graphicx}
 \usepackage{xargs}[2008/03/08]
 \usepackage[all]{xy}
 \usepackage[unicode=true,bookmarks=false,breaklinks=false,pdfborder={0 0 1},backref=section,colorlinks=false]{hyperref}
 \hypersetup{plainpages=false,colorlinks,hypertexnames=false,citecolor=blue}
 \usepackage{graphicx}
 \usepackage{graphics}
 \usepackage{epstopdf}
 \usepackage{epsfig}\DeclareGraphicsExtensions{.eps}
 \usepackage[all]{xy}
 \usepackage{textcomp}
 \usepackage{gensymb}\usepackage{eurosym}
 \usepackage{enumitem}
 \setlist[itemize]{leftmargin=*}
 \usepackage{caption}
 \usepackage{subcaption}
 \usepackage{amscd}
 \usepackage{amsfonts}
 \usepackage{mathrsfs}
 \usepackage{graphics}
 \usepackage[mathscr]{euscript}
 \usepackage{mathrsfs}
 \usepackage[normalem]{ulem}
\numberwithin{equation}{section}
\numberwithin{figure}{section}
\theoremstyle{plain}
\newtheorem{thm}{\protect\theoremname}[section]
\theoremstyle{remark}
\newtheoremstyle{myclaim}{2mm}{2mm}{}{}{\bfseries}{}{ }{\thmnumber{#2}.\thmnote{ #3. }} 
\theoremstyle{myclaim} 
\newtheorem{claim}[thm]{}
\theoremstyle{plain}
\newtheorem{lem}[thm]{\protect\lemmaname}

\newtheoremstyle{myrem}{2mm}{2mm}{\normalfont}{}{\bfseries Remark }{}{ }{\thmnumber{#2}.\thmnote{ #3. }} 
\theoremstyle{myrem} 
\newtheorem{rem}[thm]{\protect\remarkname}
\theoremstyle{definition}
\newtheorem{defn}[thm]{\protect\definitionname}
\theoremstyle{definition}
\newtheorem{example}[thm]{\protect\examplename}
\theoremstyle{plain}
\newtheorem{cor}[thm]{\protect\corollaryname}

\makeatletter
\@ifundefined{date}{}{\date{}}
\numberwithin{figure}{section}

\DeclareMathOperator{\Id}{Id}
\DeclareMathOperator{\id}{Id}

\DeclareMathOperator{\Ker}{Ker}

\DeclareMathOperator{\coker}{Coker}

\def\ot{\otimes}

\def\r{\rho}

\def\De{\Delta}

\def\M{{\mathscr{M}}}\def\T{{\mathscr{T}}}

\def\I{\mathbb{I}}

\def\gm{{g(A,m_0)}} 
 
\def\N{\mathbb N}

\email{}
\thanks{}

\DeclareGraphicsExtensions{.pdf}

\newcommand{\restr}[1]{\raisebox{-.5ex}{$|$}_{#1}}

\newcommand{\xyR}[1]{%
\makeatletter
\xydef@\xymatrixrowsep@{#1}
\makeatother
} 
\newcommand{\xyC}[1]{%
\makeatletter
\xydef@\xymatrixcolsep@{#1}
\makeatother
} 
\def\I{\mathbb{I}}

\makeatother

\providecommand{\corollaryname}{Corollary}
\providecommand{\definitionname}{Definition}
\providecommand{\examplename}{Example}
\providecommand{\lemmaname}{Lemma}
\providecommand{\remarkname}{Remark}
\providecommand{\theoremname}{Theorem}

\global\long\def\coalg#1#2{\mathscr{C}(#1,#2)}%
\global\long\def\cat#1{\mathscr{#1}}%
\global\long\def\cosl#1#2{(#1\negmedspace\downarrow\negmedspace\mathscr{#2})}%
\global\long\def\Id#1#2{\mathrm{Id}_{#1}#2}%
\global\long\def\ot{\otimes}%
\global\long\def\Ker{\mathrm{Ker\,}}%
\global\long\def\Im{\mathrm{Im\,}}%
\global\long\def\res#1#2{#2\restr{#1}}%
\global\long\def\cti{(\widetilde{C},\iota)}%
\global\long\def\ti#1{\widetilde{#1}}%
\global\long\def\u#1{_{\langle#1\rangle}}%
\global\long\def\r#1{_{(#1)}}%
\global\long\def\vf{\varphi}%
\global\long\def\mbb#1#2{\mathbb{#1}^{#2}}%
\newcommandx\mrm[2][usedefault, addprefix=\global, 1=, 2=]{\mathrm{#1}_{X}^{#2}(A,m)}%
\global\long\def\Fv{F_{V}}%
\global\long\def\Ftv{F_{V}^{\mathscr{\T}}}%
\global\long\def\Ff{\mathscr{F}_{f}^{\mathscr{}}}%
\global\long\def\Fm{\mathscr{F}_{m}^{\mathscr{}}}%
\global\long\def\Ftf{F_{f}^{\mathscr{T}}}%
\global\long\def\Ftm{F_{m}^{\mathscr{T}}}%
\global\long\def\Hf{\mathfrak{F}_{f}}%
\global\long\def\Hfp{\mathfrak{F}'_{f}}%
\global\long\def\Dfr{\mathfrak{D}_{m}}%
\global\long\def\Dmp{\mathfrak{D}'_{m}}%
\global\long\def\Dm{\mathscr{D}_{m}^{\mathscr{}}}%
\global\long\def\Dtm{D_{m}^{\mathscr{T}}}%
\global\long\def\jm{\jmath_{m}}%
\global\long\def\jf{\jmath_{f}}%
\global\long\def\sf{\iota_{f}}%
\global\long\def\sm{\iota_{m}}%
\global\long\def\pf{ \pi_{f}}%
\global\long\def\pv{ \pi_{V}^{\T}}%
\global\long\def\pm{ \pi_{m}}%
\global\long\def\al#1{\alpha{}_{#1}}%
\global\long\def\alv{\alpha{}_{V}^{\mathscr{T}}}%
\global\long\def\alf{\alpha_{f}}%
\global\long\def\alm{\alpha_{m}}%
\global\long\def\be#1{\beta_{#1}}%
\global\long\def\bev{\beta_{V}^{\mathscr{T}}}%
\global\long\def\bef{\beta_{f}}%
\global\long\def\bem{\beta_{m}}%
\global\long\def\gf{\gamma_{f}}%
\global\long\def\gm{\gamma_{m}}%

\begin{document}
\title{On deformation theory of associative algebras in monoidal categories}
\author{Abdenacer Makhlouf}
\address{\hspace*{-4mm}IRIMAS--Department of Mathematics, University of Haute--Alsace,
68093 Mulhouse, France}
\email{Abdenacer.Makhlouf@uha.fr}
\author{Drago\c{s} \c{S}tefan}
\address{\hspace*{-4mm}University of Bucharest, Department of Mathematics,
14 Academiei Street, Bucharest Ro-010014, Romania}
\email{dragos.stefan@fmi.unibuc.ro}
\begin{abstract}
We extend the classical concept of deformation of an associative algebra,
as introduced by Gerstenhaber, by using monoidal linear categories
and cocommutative coalgebras as foundational tools. To achieve this
goal, we associate to each cocommutative coalgebra $C$ and each linear
monoidal category $\M$, a $\Bbbk$-linear monoidal category $\M_{C}$.
This construction is functorial: any coalgebra morphism $\iota:C\to\widetilde{C}$
 induces a strict monoidal functor $\iota^{*}:\M_{\widetilde{C}}\to\M_{C}$.
An $\iota$-deformation of an algebra $(A,m)$ is defined as an algebra
$(A,\widetilde{m})$ in the fiber of $\iota^{*}$ over $(A,m)$. Within
this framework, the deformations of $(A,m)$ are organized into a
presheaf, which is shown to be representable. In other words, there
exists a universal deformation satisfying a specific universal property.

It is well established that classical deformation theory is deeply
connected to Hochschild cohomology. We identify and analyze the cohomology
theory that governs $\iota$-deformations in the second part of the
paper. Additionally, several particular cases and applications of
these results are examined in detail.
\end{abstract}

\maketitle
\tableofcontents{}

\section*{Introduction }

Deforming mathematical structures is a fundamental method
in mathematics and physics, focusing on how a given object can
be smoothly modified while preserving key properties. Deformation
theory serves as a bridge between algebra, geometry, and physics,
providing profound understanding into how structures evolve and interact.
The origins of deformation theory trace back to the works of Fröhlicher-Nijenhuis \cite{FN}
and Kodaira-Spencer. However, it was Gerstenhaber who developed   a systematic method for studying deformations of associative algebras, by extending the base field of an algebra to formal power series, see \cite{G1,G2,G3,G4} and \cite{GS1,GS2,GS3}.

According to \cite[p.21]{GS2}, a deformation of $A$ is given by
an ``associative'' formal power series: 
\begin{equation*}
m_{t}:=\sum_{n=0}^{\infty}m_{n}t^{n},
\end{equation*}
where each coefficient $m_{n}:A\times A\to A$ is $\Bbbk$-bilinear
and $m_{0}=m$. The associativity of $m_{t}$ means that the condition:
\begin{equation*}
\sum_{i+j=n}m_{i}\circ(m_{j}\ot_{\Bbbk}id_A)=\sum_{i+j=n}m_{i}\circ(id_A\ot_{\Bbbk}m_{j})
\label{eq:def_Gerstenhaber}\tag{AC}
\end{equation*}
must hold for all $n\geq0$.

Gerstenhaber\textquoteright s key discovery was the connection between
the deformation theory of associative algebras  and Hochschild cohomology.
Specifically, the second Hochschild cohomology group $\mathrm{HH}^{2}(A,A)$ 
plays a central role. The first-order deformation term $m_{1}$ is
always a 2-cocycle, and first-order deformations differing by a boundary
are equivalent. Consequently, the cohomology class of $m_{1}$ in
$\mathrm{HH}^{2}(A,A)$ determines the equivalence class of a deformation.
Furthermore, if a certain element in $\mathrm{HH}^{3}(A,A)$, called
the obstruction to deformation, is nonzero, then it is not possible
to extend a deformation to a higher-order one. If the second cohomology
group vanishes, the algebra is said to be rigid, meaning it admits
no non-trivial deformations.

Gerstenhaber's approach to study deformations of associative algebras has been extended
to a wide range of algebraic structures. These include algebra diagrams (presheaves
of algebras on a category), Lie algebras \cite{NR1,NR2,BM}, coalgebras,
Hopf algebras and more. A thorough overview of deformation theory
and cohomology can be found in \cite{GS2}, covering topics such as
jump deformations, rigidity, and the Gerstenhaber-Schack complex.

Building on Schlessinger's seminal work \cite{Sc}, Gerstenhaber\textquoteright s
method was further generalized by replacing the formal power series
ring $\Bbbk[[t]]$ with an arbitrary commutative associative algebra.
In this generalized framework, the deformed multiplication is defined
on $B\ot A$; see, for example, \cite{Fi,FF}. Furthermore, deformations 
can be represented by a functor $\mathrm{Def}_{g}$ from the
category of local Artinian $\Bbbk$-algebras with residue field $\Bbbk$
to the category of sets \cite{ELS}. Specific deformation theories
have also been developed for abelian categories \cite{LVdB},
monoidal 2-categories \cite{DHL} and prestacks (via Maurer-Cartan
equations) \cite{DL1}.

Among the significant applications of deformation theory is the construction
of quantized universal enveloping algebras, which underpin the theory
of quantum groups. In physics, algebraic deformations find crucial
applications in quantization, particularly through methods like deformation
quantization. Here, the algebraic structure of a classical physical
system is deformed into a quantum structure. Rather than working directly
with operators, this approach modifies the classical algebra (e.g.,
the Poisson bracket) to yield a quantum structure, such as the Moyal
bracket.

In this article, we present a new approach to the deformation theory
of associative algebras from a new perspective, using the framework
of monoidal category theory in combination with coalgebraic methods.
To ease the exposition of our results, we need some terminology and
notational conventions. Throughout  this paper, $\cat V$ represents the category
of vector spaces over a field $\Bbbk$. For a category $\mathfrak C$,
we denote the set of morphisms from $X$ to $Y$ by $\mathfrak C(X,Y)$.
The linear space of polynomials $\Bbbk[t]$ is thought of as a cocommutative
coalgebra with the comultiplication: 
\[
\Delta(t^{n})=\sum_{i=0}^{n}t^{i}\ot_{\Bbbk}t^{n-i}.
\]
 For vector spaces $X$ and $Y$, we can identify $\cat V_{t}(X,Y):=\cat V(\Bbbk[t],\cat V(X,Y))$
with $\cat V(X,Y)[[t]]$, the space of formal power series in $t$
with coefficients in $\cat V(X,Y)$. Under this identification, a
series $\sum f_{n}t^{n}$ corresponds to a linear map $f$ if and
only if $f(t^{n}):=f_{n}$ for all $n\in\N$. We define a $\Bbbk$-linear
monoidal category $\cat V_{t}$, in which objects are vector spaces
and morphisms are linear maps $f:\Bbbk[t]\to\cat V(X,Y)$, respectively.
The composition and tensor product structures in $\cat V_{t}$ are
analogous to the usual convolution product on the dual of a coalgebra.
More precisely, we have:
\begin{equation*}
(f'*f'')(t^{n}):=\sum_{i+j=n}f'(t^{i})\circ f''(t^{j})\qquad\text{and}\qquad(f_{1}\ot f_{2})(t^{n}):=\sum_{i+j=n}f_{1}(t^{i})\ot_{\Bbbk}fg_{2}(t^{j}),
\end{equation*}
for any $f'\in\cat V(X,Y)$, $f''\in\cat V(Z,X)$ and $f_{i}\in\cat V(X_{i},Y_{i})$.

We will use the term algebra in a monoidal category (such as $\cat V$
and $\cat V_{t}$) to refer to an object equipped with a multiplication
that is always associative (even if it lacks a unit). 

The cornerstone of our approach is the observation that a deformation
of an algebra $(A,m)$ in $\cat V$ can be equivalently viewed as
an algebra in $\cat V_{t}$: a formal power series $m_{t}$ satisfies
the associativity condition (AC) and $m_{0}=m$,
if and only if $(A,\ti m)$ is an algebra in $\cat V_{t}$ and $\ti m(1)=m$. 

This reformulation shifts the focus of deformation theory to the study
of algebras in monoidal categories. Our main objective is to extend
Gerstenhaber's theory to a broader context. For a cocommutative coalgebra
$C$ and a linear monoidal category $\M$, we construct a $\Bbbk$-linear
monoidal category $\mathcal{M}_{C}$ with the same objects as $\M$,
 but hom-spaces
\[
\M_{C}(X,Y)=\cat V(C,\M(X,Y)).
\]
The construction is functorial: any coalgebra morphism $\iota:C\to\ti C$
induces a strict monoidal functor $\iota^{*}$ from $\mathcal{M}_{\ti C}$
to $\mathcal{M}_{C}$, acting as the identity on objects and mapping
$f$ to $f\circ\iota$ for any morphism $f$. Notably, $\cat V_{\Bbbk}=\cat V$
and $\cat V_{\Bbbk[t]}=\cat V_{t}$, so we recover the classical case
through inclusion $\iota:\Bbbk\to\Bbbk[t]$.

In this work, we define an $\iota$-deformation of an algebra $(A,m)$
in $\M_{C}$, as an algebra $(A,\ti m)$ in $\M_{\ti C}$, belonging
to the fiber of $\iota^{*}$ over $(A,m)$. Our primary objective
is to establish that the deformations of an algebra in $\M_{C}$ define
a representable presheaf. In other words, we aim to prove the existence
of a deformation which is universal in a certain sense.

To enhance generality and applicability, we assume that all coalgebras
considered belong to a specific admissible class of coalgebras, denoted
$\T$; see Definition \ref{def:T}. Coalgebras in $\T$ should be
understood as sharing a set of specific common properties. For instance,
the class of cocommutative coalgebras provides a canonical example,
as does the class of cocommutative and pointed coalgebras. Admissible
classes are of interest because, when working within $\T$, the coalgebraic
counterpart of the universal deformations to be constructed will automatically
belong to $\T$. 

A key tool in our investigation of deformations is the notion of $\iota$-factorizations
for a linear map $f\in\cat V(C,V)$. Specifically, an $\iota$-factorization
of $f$ is a linear function $\ti f:\ti C\to V$, such that $\ti f\circ\iota=f$.
Note that any $\iota$-deformation $(A,\ti m)$ of $(A,m)$' can be
viewed as an $\iota$-factorization of $m$. Keeping in mind our goal,
we define a presheaf $\mathcal{F}_{f}:\cosl CT\to\mathcal{S}$, where
\[
\Ff\cti:=\{\ti f\in\cat V(\ti C,V)\mid\ti f\circ\iota=f\},
\]
and $\mathcal{F}_{f}$ acts on morphisms in a canonical way. Here,
$\cat S$ and $\cosl CT$ denote the category of sets and the coslice
category of $\T$ under $C$, respectively. Recall that the objects
in $\cosl CT$ are pairs $(\ti C,\iota)$, with $\ti C$ a coalgebra
in $\T$ and $\iota$ an arbitrary coalgebra morphism from $C$ to
$\ti C$. The morphisms in this category are defined in a natural
manner.

Our first significant result, Theorem \ref{prop:F=00003DH}, states
that $\Ff$ is representable, i.e., there exists $(\Ftf,\sf)$ in $\cosl CT$
such that $\Ff$ and $\cosl C{\T}(-,(\Ftf,\sf))$ are isomorphic.
An essential step of the proof is to show that, for any vector space
$V$, there exists a cofree coalgebra of type $\T$ over $V$. This
is a coalgebra $\Ftv$ in $\T$, together with a linear map $\pv$ from
$\Ftv$ to $V$, satisfying the universal property of cofree coalgebra,
for all coalgebras in $\T$; see \ref{thm:terminal_coalgebra}. 

Finally, the significance of the admissible class $\T$, is further
highlighted by the observation that, if $C$ and $\ti C$ belong to
$\T$ and are, in addition, pointed or connected, then $\Ftf$ retains
the same properties, as proved in Theorem \ref{thm:D=00003DC}.

Turning back to deformations, we can proceed as above to define, for
any algebra $(A,m)$ in $\M_{C}$, a presheaf $\Dm:\cosl CT\to\cat S$.
It assigns to a pair $(\ti C,\iota)$ the set of $\iota$-deformations
of $(A,m)$. Using the results derived for factorizations, we demonstrate
the existence of a morphism $\iota_{m}:C\to\Dtm$ of coalgebras in
$\T$ such that $\Dm$ and $\cosl CT(-,(\Dtm,\iota_{m}))$ are isomorphic;
see Theorem \ref{thm:D=00003DH}, where a slightly different notation
is used.

We have already noticed the close connection between deformation theory
and homological algebra. The second part of this article focuses on
the cohomology that governs the $\iota$-deformations of an algebra
$(A,m)$ in $\M_{C}$. We now restrict our attention to cocommutative
coalgebras (that is, $\T$ is the class of cocommutative algebras)
and require that $\iota$ defines a coalgebra extension, as described
in Subsection \ref{claim:extension}. 

From the theory of coalgebra extensions, it is well known that $X:=\coker\iota$
inherits a canonical structure as both a left and a right comodule
over $C$. Since $C$ is cocommutative, the left and right coactions
coincide up to the identification $C\ot_{\Bbbk}X\simeq X\ot_{\Bbbk}C$,
allowing us to disregard the left comodule structure. Furthermore,
$\ti C\simeq C\oplus X$ and through this decomposition, the coalgebra
structure of $\ti C$ can be entirely reconstructed from the comultiplication
of $C$, the right coaction on $X$ and a specific 2-cocycle $\omega:X\rightarrow C^{\ot2}$. 

The cohomology that governs the deformations of $(A,m)$ in the current
framework is defined by constructing a complex, denoted $\mrm[C][*]$,
as detailed in Subsection \ref{claim:C*}. 

A $\iota$-deformation $(A,\tilde{m})$ of $(A,m)$ is uniquely determined
by its restriction to $X$, which we denote by $\ti m_{X}$. The generalized
Maurer-Cartan equation:
\[
d_{X}^{2}(\ti m_{X})=\zeta,
\]
forms the bridge between cohomology and deformations. Here, as in
classical deformation theory, $\zeta:X\to\M(A^{\ot3},A)$
denotes the obstruction to deformation and it is defined by means
of $m$ and $\omega$. Using generalized Maurer-Cartan equation, we
prove another central result of this paper: Theorem \ref{thm:main_result_for_extensions}
states that $\zeta$ is always a $3$-cocycle, and it is a coboundary
if and only if $\Dm\cti$ is not empty. Moreover, if there are $\iota$-deformations
of $(A,m)$, then there exists a one-to-one correspondence between
$\Dm\cti/_{\sim}$, the set of equivalence classes of $\iota$-deformations
of $(A,m)$, and $\mrm[H][2]$. 

In the final section of the article we present several applications
of these results. Particular attention is given to the case where
the extension $C\subseteq\ti C$ arises from a graded cocommutative
coalgebra $D=\bigoplus_{i\in\N}D^{i}$, setting $C:=\oplus_{i<n}D^{i}$
and $\ti C:=\oplus_{i\leq n}D^{i}$ for some $n\geq1$. The special
case $n=1$ corresponds to $\iota$-deformations that we call infinitesimal.
For further details, see Subsections \ref{claim:graded_coalg} and
\ref{rem:infinitesimal}.

If the vector space $m(D^{0})$ is one-dimensional, we say that $m$
has rank 1. The properties of deformations of rank 1 algebras are
summarized in Theorem \ref{thm:rank_1}. Another significant case,
addressed in Theorem \ref{thm:completely_reducible}, concerns coalgebra
extensions where $X$ is completely reducible, meaning that it can be expressed
as a direct sum of one-dimensional subcomodules. For example, if $D^{0}$
is pointed and cosemisimple (e.g., $D=\Bbbk[t_{1},\dots,t_{r}])$,
then the corresponding comodule $X$ is completely reducible, as shown
in Corollary \ref{cor:D=00003Dpointed_cosemi}.

As previously mentioned, in the classical case, deforming a unital
algebra results in another unital algebra. This is one reason why
we focus on studying $\iota$-deformations of algebras that are not
necessarily unital. In the context of graded cocommutative coalgebras,
we establish an analogous result in Theorem \ref{thm:unit}.

We briefly outline the structure of the article. The foundational
results that we need later on in the paper are recalled in the first
section. The category $\cat M_{C}$ is introduced and examined in
the second section. The third section delves into $\iota$-factorizations
and the corresponding universal $\cat T$-type coalgebras, including
the $\T$-type cofree coalgebra over a vector space. The $\iota$-deformations
and the attached presheaves are introduced in the fourth section.
The fifth section is dedicated to the cohomology that controls $\iota$-deformations
in the case where $\iota$ defines a cocommutative coalgebra extension.
In the final section some applications of our most significant findings
are discussed.

\section{Preliminaries }

We begin by revisiting the definitions and foundational properties
of the key concepts required for our work, including monoidal categories
and (cofree) coalgebras.

Recall that $\Bbbk$ and  $\cat V$ denote a fixed field and the category of vector spaces over $\Bbbk$, respectively. In the introduction, we also used the notation  $\mathfrak{C}(X,Y)$ for the set of morphisms from $X$ to $Y$ in a category $\mathfrak{C}$.

Given morphisms $f:X\to Y$
and $g:Y\to Z$, their composition is denoted by $g\circ f$. The
identity morphism for an object $X\in\mathfrak{C}$ is denoted either
$I_{X}$ or simply $X$. The canonical hom-functor $\mathfrak{C}(-,X)$
assigns to a object $Y$ and to a morphism $f\in\mathfrak{C}(Y,Y')$
the set $\mathfrak{C}(Y,X)$ and the map $\mathfrak{C}(f,X):\mathfrak{C}(Y',X)\to\mathfrak{C}(Y,X)$,
respectively, where $\mathfrak{C}(f,X)(u)=u\circ f$. This functor
is an example of a presheaf, i.e., a contravariant functor with target
$\cat S$, the category of sets. In a similar way one defines the
(covariant) functor $\mathfrak{C}(X,-)$. 
\begin{claim}[Monoidal (tensor) categories]
 A \emph{monoidal category} $(\cat M,\ot,\I,a,l,r)$ is a category
$\M$ equipped with a \emph{tensor product} functor $\otimes:\cat M\times\cat M\rightarrow\cat M$,
a \emph{unit} object $\I\in\cat M$ and natural isomorphisms $a_{X,Y,Z}:(X\otimes Y)\otimes Z\rightarrow X\otimes(Y\otimes Z)$, 
$l_{X}:\I\otimes X\rightarrow X$ and $r_{X}:X\otimes\I\rightarrow X$. 
The morphism $a$ is called the \emph{associativity constraint}, while
$l$ and $r$ are called the \emph{left} and the \emph{right unit
constraints}, respectively. These constraints satisfy two coherence
conditions: the \emph{Pentagon Axiom} (ensuring associativity) and
the \emph{Triangle Axiom} (ensuring compatibility of the unit constraints).
For any objects $X,Y,Z,U$ in $\cat M$, the following commutative
diagrams illustrate these axioms. 
\begin{figure}[H]
\includegraphics{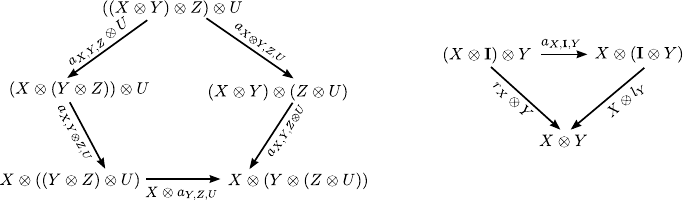}
\end{figure}
\end{claim}

\begin{claim}[The Coherence Theorem]
\label{cl:CohThm} A monoidal category is said to be \emph{strict}
if its constraints are identity morphisms. Accordingly to the Coherence
Theorem, due to S. Mac Lane, every monoidal category can be embedded
into a strict one. Following \cite[p. 420]{Maj2}, this property can
be explained in an equivalent way and used in practice as follows.
By the Pentagon Axiom, the two morphisms from $((X\otimes Y)\otimes Z)\otimes U$
to $X\otimes(Y\otimes(Z\otimes U))$ that can be written as composition
of associativity constraints coincide. The Coherence Theorem, states
that a similar property holds for all tensor products of finitely
many objects in $\cat M$: Two arbitrary compositions of associativity
constraints having the same domain and the same codomain coincide.
Therefore, we can always write $X_{1}\otimes\cdots\otimes X_{n}$
for any object obtained from $X_{1},\ldots,X_{n}$ by using $\otimes$
and brackets. Throughout this paper, to make a choice, we will use
the bracketing from the left to the right, so for any $n$ we have:
\[
X_{1}\otimes\cdots\otimes X_{n}=(X_{1}\otimes X_{2}\ot\cdots\ot X_{n-1})\otimes X_{n}.
\]
In particular, if all objects $X_{i}$ are equal to a given one, say
$X$, then for their tensor product we will use the notation $X^{\ot n}$.
Thus, $X^{\ot n}=X^{\ot n-1}\ot X$.

Also as a consequence of the Coherence Theorem, the morphisms $a$, 
$l$,  $r$ can be omitted in any computation that involves morphisms
in $\cat M$. For simplicity, we will include associativity and unit
constraints in computations only when strictly necessary.

In this paper we will work with $\Bbbk$-linear monoidal categories,
where $\Bbbk$ is a given field. Recall that $(\cat M,\ot,\I,a,l,r)$
is $\Bbbk$-linear if and only if $\M(X,Y)$ is a $\Bbbk$-linear
space, for any objects $X$ and $Y$, and the composition and the
tensor products of morphisms in $\M$ are defined by $\Bbbk$-bilinear
maps.

For more details on monoidal categories,  we refer to \cite{AMS,Ka,McL}.
\end{claim}

\begin{claim}[Coalgebras]
 All coalgebras that we are interested in are defined over $\Bbbk$.
The category of these coalgebras is denoted by $\cat C$. For a coalgebra
$(C,\Delta,\varepsilon)$ the iterated comultiplications $\De^{p-1}:C\to C^{\ot p}$
are defined recursively such that $\De^{1}:=\De$ and, for $p>1$, 
we have: 
\[
\De^{p}:=\big(\De\ot C^{\ot p-1}\big)\circ\De^{p-1}.
\]
Note that, as consequence of the fact that $\De$ is coassociative,
$\Delta^{p}=\big(C^{\ot(i-1)}\ot\De\ot C^{\ot(p-i)}\big)\circ\De^{p-1}$,
for all $0<i\leq n$. We will use Sweedler's notation
\[
\De^{p-1}(c)=\sum c\r 1\ot\cdots\ot c\r p.
\]
The linear dual of a coalgebra $C$ will be denoted by $C^{*}$. It
is a $\Bbbk$-algebra with respect to the \emph{convolution product},
defined by:
\[
(\alpha*\beta)(c)=\sum\alpha(c\r 1)\beta(c\r 2),
\]
for any $\alpha,\beta\in C^{*}$ and $c\in C$. 

Recall that a \emph{coalgebra filtration} on $C$ is an increasing
exhaustive sequence $\{C_{n}\}_{n\in\N}$ of linear subspaces (i.e.,
$C_{n}\subseteq C_{n+1}$ for all $n$ and $\bigcup_{n\in\N}C_{n}=C$),
that satisfies the relation: 
\[
\Delta(C_{n})\subseteq\sum_{i=0}^{n}C_{i}\ot C_{n-i}.
\]
The \emph{coradical filtration} of a coalgebra $C$ is an important
example of coalgebra filtration. The bottom layer is $C_{0}$, the
\emph{coradical} of $C$, that is the sum of simple subcoalgebras.
Recall that a coalgebra is called \emph{simple} if it has precisely
two subcoalgebras, the trivial one and itself. Note that $0$ is not
a simple coalgebra. The other terms of the coradical filtration are
defined inductively by taking ${C}_{n+1}=\Delta^{-1}({C}\otimes{C}_{n}+{C}_{0}\otimes{C})$, 
for all $n\in\mathbb{N}$. 

Of course, if the dimension of $C$ is $1$, then $C$ is simple and
it is spanned by a unique \emph{group-like element} (i.e., an element
$x\in{C}$ satisfying $\Delta(x)=x\otimes x$ and $\varepsilon_{C}(x)=1$).
The set of group-like elements is denoted by ${G(C)}$. In the case
when $C_{0}$ is generated as a vector space by $G(C)$, the coalgebra
$C$ is called \emph{pointed}. By definition $C$ is \emph{connected}
if and only if $C_{0}$ is $1$-dimensional. For a connected coalgebra
$C$, the unique group-like element is denoted by $x_{C}$.
\end{claim}

Graded coalgebras provide important examples of coalgebra filtrations.
A coalgebra $C$ is \emph{graded} if there is a family $\{C^{k}\}_{k\in\N}$
of linear subspaces such that $C=\oplus_{n\in\N}C^{n}$, the restriction
of $\varepsilon_{C}$ to $C_{+}=\oplus_{n>0}C_{n}$ is equal to $0$,
and 
\[
\Delta(C_{n})\subseteq\bigoplus_{i=0}^{n}C^{i}\ot C^{n-i}.
\]
The subspace $C^{n}$ is called the \emph{homogeneous component} of
degree $n$. The component $C^{0}$ is a subcoalgebra, so the canonical
maps $\iota:C^{0}\to C$ and $\lambda:C\to C^{0}$ are coalgebra morphisms. 

For a $p$-tuple of nonnegative integers $I=(i_{1},\dots,i_{p})$,
let $|I|:=\sum_{j=1}^{p}i_{j}$ and $C^{I}:=C^{i_{1}}\ot\cdots\ot C^{i_{p}}$.
If $c\in C^{n}$, then there are unique elements $\Delta_{I}^{p-1}(c)\in C^{I}$
such that $\Delta^{p-1}(c)=\sum_{|I|=n}\Delta_{I}^{p-1}(c)$. We will
use a Sweedler-like notation:
\[
\Delta_{I}^{p-1}(c)=\sum c\r{1,i_{1}}\ot\cdots\ot c\r{p,i_{p}},
\]
where $i_{j}$ indicates that the $j^{\text{th}}$ factor of the tensor
monomial is homogeneous of degree $i_{j}$.

Every graded coalgebra $C=\oplus_{n\in\N}C^{n}$ admits a canonical
filtration $\{C_{\leq n}\}_{n\in\N}$, where $C_{\leq n}:=\sum_{i=0}^{n}C^{i}$.
In the case when this filtration and the coradical filtration of $C$
coincide, we will say that C is \emph{coradically graded}. 
\begin{claim}[The cofree coalgebra over a vector space]
\label{fa:free-coalgebra} Let $V$ be a $\Bbbk$-linear space. Recall
that the cofree coalgebra over $V$ is a coalgebra $F_{V}$ equipped
with a linear map $\pi_{V}:F_{V}\to V$ satisfying the following universal
property: If $D$ is an arbitrary coalgebra and $h:D\to V$ is a morphism
of vector spaces, then there exists a unique coalgebra map $\al V(D)(h):D\to F_{V}$
such that the following diagram is commutative.
\begin{equation}
\begin{array}{c}
\xymatrix{D\ar[rr]^{h}\ar@{..>}[rd]_{\al V(D)(h)} &  & V\\
 & F_{V}\ar[ru]_{\pi_{V}}
}
\end{array}\label{diag:universal}
\end{equation}
Equivalently, the presheaf $\cat V(-,V):\cat C\to\cat S$ is representable,
meaning that it is isomorphic to $\coalg -{\Fv}$. The isomorphism
is given by the natural linear map $\al V:=\{\al V(D)\}_{D\in\cat C}$
and its inverse $\be V:=\{\be V(D)\}_{C\in\cat C}$, where: 
\begin{equation}
\be V(D):\cat C(D,F_{V})\to\cat V(D,V),\quad\be V(D)(\varphi)=\pi_{V}\circ\varphi.\label{eq: beta(C)}
\end{equation}
For the construction of the free coalgebra and an overview of its
properties, see \cite{Haz}.
\end{claim}

\section{The monoidal category \texorpdfstring{$(\M_{C},\protect\ot,\I)$}{MC}\label{sec:TM_C}}

Let $(\M,\ot,\I,a,l,r)$ be a $\Bbbk$-linear monoidal category. In
this section, we construct a new $\Bbbk$-linear category $\M_{C}$
for an arbitrary nonzero coalgebra $C$. Furthermore, if $C$ is cocommutative,
$\M_{C}$ inherits a monoidal structure. Our approach to the deformation
theory of associative algebras relies on studying algebras in $\M_{C}$,
for various choices of $C\in\cat C$.
\begin{claim}[The monoidal category $(\M_{C},\ot,\I,a,l,r)$]
\label{fact:MC} By definition, the objects of $\M_{C}$ coincide
with those of $\M$. However, for objects $X$ and $Y$ in $\M_{C}$,
we let:
\[
\M_{C}(X,Y):=\cat V(C,\M(X,Y)).
\]
The composition of two morphisms $f\in\M_{C}(X,Y)$ and $g\in\M_{C}(Y,Z)$
is defined by: 
\[
(g\ast f)(x):=\sum g(x\r 1)\circ f(x\r 2).
\]
Since $\M$ and $\M_{C}$ are linear monoidal categories, $A:=\cat M(X,X)$
and $\M_{C}(X,X)=\cat V(C,A)$ are associative algebras, with their
products defined by the composition operation in $\M$ and $\M_{C}$,
respectively. By definition, the product in the latter algebra corresponds
to the convolution product $\cat V(C,A)$, cf. \cite[\S1.2]{Mo}.
Analogously, we refer to the composition operation in $\M_{C}$ as
the \emph{convolution composition}. 

Clearly, $\M_{C}$ is a category with respect to the convolution composition.
The identity morphism $\Id X{}:C\to\M(X,X)$ is defined by $\id_{X}(c):=\varepsilon_{C}(c)I_{X}$.
Recall that for the identity of $X$, viewed as object in $\M$, we
use the notation $I_{X}$ or, simply, $X$.

Every morphism of nonzero coalgebras $\iota:C\to\widetilde{C}$ induces
a functor $\iota^{*}:\M_{\widetilde{C}}\to\M_{C}$ such that $\iota^{*}(X)=X$
and $\iota^{*}(f)=f\circ\iota$, for any $f\in\M_{\widetilde{C}}(X,Y)$.
For example, the counit $\varepsilon_{C}:C\to\Bbbk$ is a coalgebra
map. Since $\M_{\Bbbk}$ is canonically isomorphic to $\M$, we can
embed $\M$ into $\M_{C}$ via the faithful functor $\varepsilon_{C}^{*}$.
Note that for any morphism $f$ in $\M$ and $c\in C$ we have $\varepsilon_{C}^{*}(f)(c)=\varepsilon_{C}(c)f$.
We emphasize that the coalgebra $C$ is assumed to be nonzero when
working with the category $\M_{C}$. Thus, $f\to\varepsilon_{C}^{*}(f)$ is
injective indeed. 

In order to define a monoidal structure on $\M_{C}$, we have to assume
that $C$ is cocommutative. The tensor product of two objects and
the unit object in $\M_{C}$ coincide with those in $\M$. The associativity
constraint and the unit constraints in $\M_{C}$ are the morphisms
$\varepsilon_{C}^{*}(a_{X,Y,Z})$, $\varepsilon_{C}^{*}(l_{X})$ and
$\varepsilon_{C}^{*}(r_{X})$, respectively.

The tensor product of two objects and the unit object in $\cat M_{C}$
coincide with those in $\M$. On the other hand, if $f:C\to\M(X,Y)$
and $f':C\to\M(X',Y')$ are morphisms in $\M_{C}$, then we define
the tensor product $f\ot f':C\to\M(X\ot X',Y\ot Y')$ by: 
\[
(f\ot f')(c)=\sum f(c\r 1)\ot f'(c\r 2).
\]
Note that we can define $f\ot f'$ as above, even though $C$ is not
cocommutative. But, in general, this tensor product is not a bifunctor.
However, assuming that $C$ is cocommutative, the relation: 
\[
(f\ot g)*(f'\ot g')=(f*f')\ot(g*g')
\]
holds, for all morphisms such that the convolution composition is
well-defined.

Now we can easily check that the above data define a $\Bbbk$-linear
monoidal structure on $\M_{C}$, and that $\varepsilon_{C}^{*}:\M\to\M_{C}$
is a strict monoidal functor. More generally, if $\iota:C\to\widetilde{C}$
is a morphism of coalgebras, then $\iota^{*}:\M_{\widetilde{C}}\to\M_{C}$
is a strict monoidal functor.
\end{claim}

\begin{claim}[{The particular case $C=\Bbbk[t]$}]
\label{fa:k=00005Bt=00005D} Recall that the vector space of polynomials
in one indeterminate $t$ has a canonical structure of connected coalgebra
such that $1$ is a group-like element and, for $n>0$, we have: 
\[
\De(t^{n})=\sum_{i=0}^{n}t^{i}\ot t^{n-i}.
\]
The counit is equal to zero on the subspace spanned by all powers
$t^{n}$,  with $n>0$.

Hence, for any linear monoidal category $\M$, it makes sense to consider
$\M_{\Bbbk[t]}$. Clearly, 
\[
\M_{\Bbbk[t]}(X,Y)\simeq\big\{\sum_{n=0}^{\infty}f_{n}t^{n}\mid f_{n}\in\M(X,Y)\big\},
\]
where the isomorphism maps a morphism $f$ to the formal series $\sum_{n=0}^{\infty}f(t^{n})t^{n}$.
Via this identification, the convolution composition of $f{}=\sum_{n=0}^{\infty}f_{n}t^{n}$
and $f'=\sum_{n=0}^{\infty}f'_{n}t^{n}$ is the formal series 
\[
f\ast f'=\sum_{n\geq0}\big(\sum_{i=0}^{n}f_{i}\circ f_{n-i}'\big)t^{n}.
\]
The formula above is meaningful provided that $f'_{i}\in\M(X,Y)$
and $f_{i}\in\M(Y,Z)$, for some objects $X,Y\text{ and }Z$ in $\M$,
and all $i\in\N$. Similarly, if $g=\sum_{n\geq0}g_{n}t^{n}$ and
$g'=\sum_{n\geq0}g'_{n}t^{n}$, then: 
\[
g\ot g'=\sum_{n\geq0}\big(\sum_{i=0}^{n}g_{i}\ot g_{n-i}'\big)t^{n}.
\]
Of course, $\M_{\Bbbk[t]}$ is a generalization of the category $\cat V_{t}$
mentioned in the introduction and related to the classical deformation
theory. 
\end{claim}

\begin{claim}[The congruence relation]
Throughout the remaining part of this section we fix a coalgebra
filtration $\{C_{n}\}_{n\in\N}$ on $C$. Let $\iota_{n}$ denote
the inclusion map of $C_{n}$ into $C$. Let $f,g\in\M_{C}(X,Y)$. 

We say that $f$ and $g$ are $n$-\emph{congruent}, and we write
$f{}\equiv_{n}g{}$, if and only if $f-g$ vanishes on $C_{n}$. Denoting
by $\iota_{n}$, the inclusion map of $C_{n}$ into $C$, the morphisms
$f$ and $g$ are $n$-congruent if and only if $f$ and $g$ are
identical as morphisms in $\M_{C_{n}}(X,Y)$, i.e., $\iota_{n}^{*}(f)=\iota_{n}^{*}(g)$.
The following properties of the equivalence relation $\equiv_{n}$
are straightforward to verify, and their proofs are omitted.

Let $g$ and $g'$ be $n$-congruent morphisms in $\M_{C}(Y,Z)$.
If $f{}\in\M_{C}(X,Y)$ and $h\in\M_{C}(Z,U)$, then: 
\begin{equation}
f\ast g\equiv_{n}f\ast g'\qquad\text{and}\qquad g\ast h\equiv_{n}g'\ast h.
\end{equation}
In particular, if $f$ and $g$ are convolution composable and $f\equiv_{n}0$,
then $f\ast g\equiv_{n}0$.

The relation $\equiv_{n}$ is also compatible with the tensor product
and is additive. Specifically, the following hold:
\begin{equation}
f\ot h\equiv_{n}f'\ot h,\qquad f\ot h\equiv_{n}f\ot h'\qquad\text{and}\qquad f+f''\equiv_{n}f'+f'',
\end{equation}
provided that the tensor products and the sums are well-defined, and
$f\equiv_{n}f'$ as well as $h\equiv_{n}h'$.

As a consequence of the above relations, let us show that $f\ot f'\equiv_{n+n'}0$,
for all $f\in\M_{C}(X,Y)$ and $f'\in\M_{C}(X',Y')$ such that $f\equiv_{n}0$
and $f'\equiv_{n'}0$. Indeed, if $i,i',n$ and $n'$ are nonnegative
integers such that $i+i'=n+n'$,  then either $i\leq n$ or $i'\leq n'$.
Since $\{C_{n}\}_{n\in\N}$ is a coalgebra filtration, it follows:
\[
\Delta(C_{n+n'})\subseteq\sum_{i+i'=n+n'}C_{i}\ot C_{i'}\subseteq C_{n}\ot C+C\ot C_{n'}.
\]
Thus $(f\ot f')(C_{n+n'})\subseteq f(C_{n})\ot f'(C)+f(C)\ot f'(C_{n'})=0$. 
Hence, $f\ot f'\equiv_{n+n'}0$.

Similarly, it can be shown that if $g\equiv_{n}0$ and $g'\equiv_{n'}0$,
then $g*g'\equiv_{n+n'}0$, provided $g$ and $g'$ are composable.
\end{claim}

\begin{lem}
\label{le:computation_f} Let $\phi:=\iota\circ\lambda$, where $\lambda$
is a retract of $\iota$ (i.e., a linear left inverse of $\iota$). 
\begin{enumerate}
\item[\emph{(a)}]  If $c\in C_{n+1}$ and $\phi_{i}:=\phi{}^{\ot(i-1)}\ot C\ot\phi{}^{\ot(p-i)}$,
for all $i$, then the image of $\sum_{i=1}^{p}\phi_{i}\circ\Delta^{p-1}-\Delta^{p-1}$
is included into $C_{n}^{\ot p}$.
\item[\emph{(b)}]  If $\Theta:C^{\ot p}\to W$ is a $\Bbbk$-linear map that vanishes on $C_{n}^{\ot p}$, then
\begin{equation}
\Theta\circ\Delta^{p-1}\equiv_{n+1}\sum_{i=1}^{p}\Theta\circ\phi_{i}\circ\Delta^{p-1}.\label{eq:Theta}
\end{equation}
\end{enumerate}
\end{lem}

\begin{proof}
Since $\{C_{n}\}_{n\in\N}$ is a coalgebra filtration, we have $\Delta^{p-1}(c)=\sum_{j=1}^{p}c^{j}+c'$, 
where $c^{j}\in C_{0}^{\ot(j-1)}\ot C_{n+1}\ot C_{0}^{\ot(p-j)}$
and $c'\in C_{n}^{\ot p}$. Therefore,

\[
\sum_{i=1}^{p}\phi_{i}\Delta^{p-1}(c)=\sum_{i=1}^{p}\sum_{j=1}^{p}\phi_{i}(c^{j})+\sum_{i=1}^{p}\phi_{i}(c')=\sum_{i=1}^{p}\phi_{i}(c^{i})+\sum_{i\neq j}\phi_{i}(c^{j})+\sum_{i=1}^{p}\phi_{i}(c').
\]
Since $\phi(x)=x$,  for all $x\in C_{0}$, we get $\phi_{i}(c^{i})=c^{i}$.
Thus,
\[
\sum_{i=1}^{p}\phi_{i}\Delta^{p-1}(c)-\Delta^{p-1}(c)=\sum_{i=1}^{p}\phi_{i}\Delta^{p-1}(c)-\sum_{i=1}^{p}\phi_{i}(c^{i})-c'=\sum_{i\neq j}\phi_{i}(c^{j})+\sum_{i=1}^{p}\phi_{i}(c')-c'.
\]
The image of $\phi$ is equal to $C_{0}$,  so $\phi_{i}(c^{j})\in C_{0}^{\ot p}$,
for any $i\neq j$. Similarly, since $C_{0}\subseteq C_{n}$, it follows
that $\phi_{i}(c')$ is an element in $C_{n}^{\ot p}$ for any $i$.
Thus part (a) is proved. The second assertion is obvious, considering
the definition of the congruence relation.
\end{proof}
\begin{rem}
\label{re:Theta} For any $c\in C_{n+1}$, the relation \eqref{eq:Theta}
can be equivalently written as follows:
\begin{equation}
\Theta(\sum c\r 1\ot\cdots\ot c\r p)=\sum_{i=1}^{p}\Theta\big(\sum\phi(c\r 1)\ot\cdots\ot\phi(c\r{i-1})\ot c\r i\ot\phi(c\r{i+1})\ot\cdots\ot\phi(c\r p)\big)\label{eq:Theta'}.
\end{equation}
\end{rem}

\begin{claim}
\label{fa:iTakeuchi Lemma} Takeuchi Lemma states that a linear map
$f$ from a $\Bbbk$-coalgebra $D$ to a $\Bbbk$-algebra $A$ is
invertible in convolution if and only if the restriction of $f$ to
the coradical of $D$ is invertible. We need a slight improvement
of this result for the case where $f$ is a morphism in $\M_{D}(X,Y)$
and the coradical of $D$ is replaced with the bottom term of a coalgebra
filtration $\{D_{n}\}_{n\in\N}$ on $D$.

First of all, adapting the proof of \cite[Lemma 5.2.10]{Mo}, for
$f$ as above, let us show that $f'':=f\res{D''}{}$ is invertible
in $\M_{D''}$ if and only if $f':=f\res{D'}{}$ is invertible in
$\M_{D'}$. Here $D'\subseteq D''$ are subcoalgebras of $D$ such
that
\begin{equation}
\Delta(D'')\subseteq D'\ot D''+D''\ot D'.\label{eq:ext}
\end{equation}
The direct implication is straightforward. For the converse, let us
assume that $f'$ is invertible in $\M_{D'}$. Let $g'$ be the inverse
of $f'$. We pick a linear map $g'':D''\to\M(Y,X)$ which extends
$g'$ and then we set $h:=\Id Y{}-f''*g''\in\M_{D''}(Y,Y)$. Clearly,
$h\res{D'}{}$ is equal to zero since $f'$ and $g'$ are inverses
each other and $g''$ extends $g'$. Therefore, in view of Equation
\eqref{eq:ext}, it follows that $h*h=0$. Thus $\Id Y{}+h$ is the
inverse of $f''*g''=\Id Y{}-h$. In conclusion, $g''*(\Id Y{}+h)$
is a right inverse of $f''$ in $\M_{D''}$. Similarly, it can be
shown that $f''$ is left invertible in $\M_{D''}$.

Next, we consider a coalgebra filtration $\{D_{n}\}_{n\in\N}$ on
$D$ such that $f\res{D_{0}}{}$ is invertible in $\M_{D_{0}}$. From
the definition of a coalgebra filtration, Equation \eqref{eq:ext}
holds for $D':=D_{n}$ and $D'':=D_{n+1}$. Hence, by induction, $f\res{D_{n}}{}$
is invertible in $\M_{D_{n}}$, for all $n\in\N$. Therefore, we conclude
that $f$ is invertible.

We refer to the result that we have just proved still as Takeuchi
Lemma.
\end{claim}

\begin{claim}[Invertible morphisms in $\M_{C}$]
\label{claim:inverse_Id+f}Let $f\in\M_{C}(X,X)$ be a morphism such
that $f\equiv_{n}0$, for some $n$. By Takeuchi Lemma, $\id_{X}+f$
is invertible in $\M_{C}$, since $\id_{X}+f\equiv_{0}\id_{X}$. We
claim that,
\begin{equation}
(\id_{X}+f)^{-1}\equiv_{n+1}\id_{X}-f.\label{eq:u_inverse}
\end{equation}
Indeed, let $g$ denotes the inverse of $\id_{X}+f$. Equation \eqref{eq:u_inverse}
is equivalent with $\Id X{}-\iota_{n+1}^{*}(f)=\iota_{n+1}^{*}(g)$.
On the other hand, $\Id X{}-f=(\Id X{}-f*f)*g$. Applying the functor
$\iota_{n+1}^{*}$ to both side of the previous identity, we obtain:
\[
\Id X{}-\iota_{n+1}^{*}(f)=(\Id X{}-\iota_{n+1}^{*}(f*f))*\iota_{n+1}^{*}(g).
\]
Therefore, it is sufficient to verify that $f*f\equiv_{n+1}0$. First
let us consider the case where $n=0$. 

We keep the notation introduced in Lemma \ref{le:computation_f}.
In particular, $\lambda$ is a retract of $\iota$, and
$\phi=\iota\circ\lambda$. Let $\Theta:C\ot C\to\M_{C}(X,X)$ be the
map defined by $\Theta(c\ot c')=f(c)\circ f(c')$. For $c\in C_{1}$,
Equation \eqref{eq:Theta'} yields us: 
\[
(f*f)(c)=\Theta(\Delta(c))=\sum f(\phi(c\r 1))\circ f(c\r 2)+\sum f(c\r 1)\circ f(\phi(c\r 2))=0,
\]
where for deriving the last equation we used that $f$ vanishes on
$C_{0}$. 

Let $n>0$.  By definition of coalgebra filtrations, $\Delta(C_{2n})\subseteq\sum_{i=0}^{2n}C_{i}\ot C_{2n-i}$.
Either $i$ or $2n-i$ is less than or equal to $n$. Since $f$ is
trivial on $C_{n}$, it follows that $f*f\equiv_{2n}0$. Thus$f*f\equiv_{n+1}0$.

Note that, if $g$ is invertible in $\M_{C}(X,Y)$ and $f\equiv_{n}0$,
then we have:
\[
(g+f)^{-1}\equiv_{n+1}g^{-1}-g^{-1}*f*g^{-1}.
\]
Indeed, $g+f=g*(\Id X{}+g^{-1}*f)$ and $g^{-1}*f\equiv_{n}0$, so
we can apply Equation \eqref{eq:u_inverse}.
\end{claim}

\section{Some constructions in the category of coalgebras}

Cofree coalgebras play a central role in our work. In this section,
we establish the existence of such coalgebras within certain suitable
categories of coalgebras. Subsequently, we use them to prove the representability
of several presheaves associated with factorizations of a linear map
and deformations of associative algebras.
\begin{defn}
\label{def:T}Recall that $\cat C$ denotes the category of $\Bbbk$-coalgebras.
A class $\T\subseteq\cat C$ is called \emph{admissible} if it contains
$\Bbbk$ (with the canonical coalgebra structure) and satisfies the
following two conditions. 
\end{defn}

\begin{enumerate}
\item[$(\dagger)$] $\T$ is closed under quotients, that is any homomorphic image of
a coalgebra in $\T$ belongs to $\T$. 
\item[$(\ddagger)$] If $\{C^{i}\}_{i}\subseteq\T$ is a family of subcoalgebras of an
arbitrary coalgebra (not necessarily in $\T$), then $\sum_{i}C^{i}\in\T$. 
\end{enumerate}
We say that $C$ is a $\T$\emph{-coalgebra} or \emph{of type} $\T$,
whenever $C\in\T$. Unless explicitly stated otherwise, the coalgebras
we consider are of type $\T$, for some given admissible class $\T$.
We regard $\T$ as a full subcategory of $\cat C$.
\begin{rem}
\label{rem:direct_sum=00003DT-type} (a) If $C$ is an arbitrary nonzero
coalgebra, then $\Bbbk$ is equal to the image of $\varepsilon_{C}$.
Thus, in the definition of admissible classes we can replace the requirement
that $\Bbbk$ is in $\T$ with the assumption that there exists a
nonzero coalgebra in $\T$. 

(b) Under the assumption that ($\dagger)$ is true, condition ($\ddagger)$
is equivalent to the fact that $\T$ is closed under direct sums.
Indeed, let $\varphi^{i}$ denote the canonical inclusion of $C^{i}$
into $\oplus_{i}C^{i}$. Then $\varphi^{i}(C^{i})$ is in $\T$, cf.
$(\dagger)$. Hence, supposing that ($\ddagger)$ holds, it follows
that $\oplus_{i}C^{i}=\sum_{i}\varphi^{i}(C^{i})$ is in $\T$. The
converse is also true, since $\sum_{i}C^{i}$ is a quotient of $\oplus_{i}C^{i}$.

(d) If necessary, we can add to an admissible class the coalgebra
$0$. Indeed, if $\T$ is admissible, then $\T\bigcup\{0\}$ is also
admissible.
\end{rem}

\begin{example}
\noindent \label{exa:T-type}Obviously, the category of all coalgebras
and the category of cocommutative coalgebras are admissible. Another
example is the category of pointed coalgebras. Indeed, by \cite[Proposition 8.0.3]{Sw},
the sum of a family of pointed coalgebras is pointed, so $(\ddagger)$
holds true. The condition $(\dagger)$ follows by \cite[Corollary 5.3.5]{Mo}.
If necessary, we can assume that $0$ belongs to this class, though
$0$ is not pointed. 

\noindent Note that $\T'\bigcap\T''$ is admissible, provided that
both $\T'$ and $\T''$ are admissible. For example, if $\T$ is admissible,
then the class of cocommutative and pointed coalgebras in $\T$ are
admissible too. 
\end{example}

Our goal now is to prove the existence of a cofree coalgebra in $\T$
over a vector space $V$, and then to discuss some application of
this result.
\begin{thm}
\label{thm:terminal_coalgebra} For a given vector space $V$, there
exists a cofree coalgebra $(\Ftv,\pv)$ in $\T$. In other words,
there is a $\T$-coalgebra $\Ftv$ which represent the presheaf $\cat V(-,V):\T\to\cat S$.
\end{thm}

\begin{proof}
We use the notation from subsection \ref{fa:free-coalgebra}. If $h\in\cat V(D,V)$
and $D$ is of type $\T$, then we let $\varphi_{h}:=\al V(D)(h)$.
Since $\T$ is admissible, $\Im\varphi_{h}$ is a $\T$-coalgebra.
Thus, by condition $(\ddagger)$, $\Ftv:=\sum_{h}\Im\varphi_{h}$
is also in $\T$. Let $\pv:=\pi_{V}\circ\jmath_{V}$, where $\jmath_{V}$
is the inclusion map of $\Ftv$ into $F_{V}$.  Since $\Im\varphi_{h}\subseteq\Ftv$,
there exists a unique morphism $\alv(D)(h)$ in $\cat C(D,\Ftv)$,
such that $\jmath_{V}\circ\alv(D)(h)=\varphi_{h}$. It is worth mentioning
that $\alv(D)(h)$ is the unique coalgebra morphism such that 
\begin{equation}
\pv\circ\alv(D)(h)=h.\label{eq:alpha_V^T}
\end{equation}
We denote the map $h\mapsto\alv(D)(h)$ by $\alv(D)$. Let us prove
that $\bev(D)$ is an inverse of $\alv(D)$, where $\bev(D)(\varphi):=\pv\circ\varphi$,
for all $\varphi\in\cat C(D,\Ftv)$. Indeed, $\bev(D)$ is a left
inverse of $\alv(D)$, in view of Equation \eqref{eq:alpha_V^T}.
The fact that $\bev(D)$ is a right inverse of $\alv(D)$ follows
by the universal property of the cofree coalgebra and the relation
\[
\pv\circ\alv(D)(\pv\circ\varphi)=\pv\circ\varphi.
\]
Obviously, $\bev:=\{\bev(D)\}_{D\in\T}$ is a natural transformation,
so the family $\alv:=\{\alv(D)\}_{D\in\T}$ is natural too.
\end{proof}
\begin{rem}
Of course, if $\T$ is the class of all coalgebras, then $\Ftv$ is
precisely the cofree coalgebra over $V$. For this reason, we refer
to $\Ftv$ as the \emph{cofree coalgebra of type $\T$ over} $V$.
For example, taking $\T$ to be the class of all cocommutative coalgebras,
then we recover the construction of the cofree cocommutative coalgebra
over $V$. Of course, the theorem may also be applied to the category
of cocommutative or/and pointed coalgebras (in some admissible class
$\T$) to get the cofree cocommutative or/and pointed coalgebra (in
$\T$) over $V$.  
\end{rem}

We are going to define one of the main tools needed in our study of
deformations of associative algebras, namely $\iota$-factorizations
of a linear map. Roughly speaking, factorizations capture the underlying
linear structure of a deformation, obtained by disregarding the the
associativity property.
\begin{claim}[Factorizations and the presheaves $\Ff$ and $\Hf\cti$]
\label{claim:characterization_Ff} Henceforth, $\T$ will denote
an admissible class. We fix a nonzero linear map $f:C\to V$ and a
coalgebra morphism $\iota:C\to\widetilde{C}$, with $C$ and $\widetilde{C}$
coalgebras in $\T$. It is well-known that $D\to0$ is a coalgebra
morphism if and only if $D=0$ (if the zero map is a morphism, then
the counit of $D$ would vanish, which forces $D$ to be equal to
$0$). Stated differently, the image of $h$ does not vanish, for
any coalgebra morphism $h:D\to E$,  with $D\neq0$. In particular,
since $f\neq0$, it follows that $\widetilde{C}\neq0$. 
\end{claim}

\begin{defn}
\label{de:T-factorization}We say that a linear map $\ti f:\widetilde{C}\to V$
is an \emph{$\iota$-factorization} of $f$ provided that $\ti f\circ\iota=f$.
By a \emph{factorization} of $f$ we mean an $\iota$-factorization,
for a certain $\iota$ as above. Sometimes, for greater clarity, an
$\iota$-factorization $\ti f$ of $f$ will be written as a triple
$(C,\iota,\ti f)$.  For the set of $\iota$-deformations of $f$ we
use the notation: 
\begin{equation}
\Ff\cti:=\{\ti f\in\cat V(\ti C,V)\mid\ti f\circ\iota=f\}.\label{eq:functor F_f(C,i)}
\end{equation}
\end{defn}

\begin{claim}
We organize the factorizations of $f$ as a presheaf on \emph{$\cosl C{\T}$},
the \emph{coslice category of} $\T$ \emph{under} $C$. Recall that
this category consists in all pairs $\cti$ with $\widetilde{C}$
in $\T$ and $\iota\in\cat C(C,\widetilde{C})$.  The set of morphisms
from $\cti$ to $(\widetilde{C}',\iota')$ contains all coalgebra
morphism $\xi$ form $\widetilde{C}$ to $\widetilde{C}'$ such that
$\xi\circ\iota=\iota'$. 

By definition, $\Ff:\cosl CT\to\cat S$ maps $\cti$ to the set $\Ff\cti$.
For a morphism $\xi$ as above, $\Ff(\xi)$ is given by:
\[
\Ff(\xi)(f'):=f'\circ\xi,
\]
where$\ti f'$ is an $\iota'$-factorization of $f$. Note that it
may occur that $\Ff(\widetilde{C}',\iota')=\emptyset$.  In this case,
of course, $\Ff(\xi)$ is the empty function.

By Theorem \ref{thm:terminal_coalgebra}, there exists a unique morphism
of coalgebras $\sf:C\to\Ftv$ such that $\pv\circ\sf=f$.  Then, $(\Ftv,\sf)$
is an object in the coslice category of $\T$ under $C$ and $(\Ftv,\sf,\pv)$
is an example of $\sf$-factorization of $f$. Our main goal in this
section is to prove that $\Ff$ and $\Hf:=\cosl C{\T}(-,(\Ftv,\sf))$
are isomorphic. Note that
\[
\Hf\cti:=\{\varphi\in\cat C(\ti C,\Ftv)\mid\varphi\circ\iota=\sf\},\quad\Hf(\xi)(\varphi)=\varphi\circ\xi.
\]
By the universal property of $\Ftv$, a morphism of coalgebras $\vf$
from $\widetilde{C}$ to\textbf{ $\Ftv$} belongs to $\Hf\cti$ if
and only if $\pv\circ\varphi\circ\iota=\pv\circ\sf$. Therefore, in
view of Equation \eqref{eq:alpha_V^T}, we get
\begin{equation}
\Hf\cti=\{\vf\in\cat C(\ti C,\Ftv)\mid\pv\circ\varphi\circ\iota=f\}.\label{eq:phi}
\end{equation}
Now, for a given $\iota$-factorization of $f$, say $\ti f$, we
let $\vf_{\ti f}:=\alv(\widetilde{C})(\ti f)$. Then, by the definition
of $\alv(\widetilde{C})$, 
\[
\pv\circ\varphi_{\ti f}\circ\iota=\ti f\circ\iota=f.
\]
This means that $\vf_{\ti f}$ is a morphism in $\Hf\cti$, so $\alv(\widetilde{C})$
maps $\Ff\cti$ to $\Hf\cti$.  On the other hand, for any $\varphi\in\Hf\cti$,
\[
\bev(\widetilde{C})(\varphi)\circ\iota=\pv\circ\varphi\circ\iota=f.
\]
In other words, $\bev(\widetilde{C})(\varphi)$ is an element in $\Ff\cti$.
We deduce that one of the sets $\Ff\cti$ and $\Hf\cti$ is empty
if and only if the other is so. Let $\alf\cti$ and $\bef\cti$ denote
the restrictions of $\alv(\widetilde{C})$ and $\bev(\widetilde{C})$
to $\Ff\cti$ and $\Hf\cti$, respectively. It follows that these
maps are inverses each other and, clearly, $\bef:=\{\bef\cti\}_{\cti\in\cosl C{\T}}$
is a natural transformation. Summarizing, we proved the following
theorem.
\end{claim}

\begin{thm}
\label{prop:F=00003DH}The presheaves $\Ff$ and $\Hf$ are isomorphic,
i.e., $\Ff$ is representable.
\end{thm}

\begin{claim}[The coalgebra $\Ftf\cti$]
\label{claim:C_f} We will further optimize the construction of the
presheaf $\Hf$ by reducing the range of the morphisms in $\Hf\cti$
as much as possible. To do that, for every object $\cti$ in $\cosl CT$,
we associate a $\T$-coalgebra $\Ftf\cti$ as follows. If there are
no $\iota$-factorizations of $f$, we set $\Ftf\cti=0$. Otherwise,
we let:

\begin{equation}
\Ftf\cti:=\sum_{\ti f\in\Ff\cti}\Im\varphi_{\ti f}.\label{eq:F_f(C,i)}
\end{equation}
We have seen in Subsection \ref{claim:characterization_Ff} that $\Im\vf_{\ti f}\neq0$
for any $\ti f\in\Ff\cti$. We conclude that $\Ff\cti=\emptyset$
if and only if $\Ftf\cti=0$.

Let $\jf\cti$ and $\pf\cti$ denote the inclusion map of $\Ftf\cti$
in $\Ftv$ and the restriction of $\pv$ to $\Ftf\cti$, respectively.
Last but not the least, we let
\begin{equation}
\Hfp\cti:=\{\vf'\in\cat C(\ti C,\Ftf\cti)\mid\pf\cti\circ\vf'\circ\iota=f\},\label{eq: F_and_Fi}
\end{equation}
whenever there are $\iota$-factorizations of $f$. If $\Ff\cti$
is empty, then $\Hfp\cti:=\emptyset$. 
\end{claim}

\begin{thm}
\label{te:relative cofree} For every $\cti$ in $\cosl C{\T}$, there
is a bijective correspondence $\Hfp\cti\stackrel{}{\xleftrightarrow{1-1}}\Hf\cti$.
\end{thm}

\begin{proof}
If $\vf'$ is a coalgebra morphism in $\Hfp\cti$, then $\jmath_{f}\cti\circ\vf'$
is an element in $\Hf\cti$. Therefore, $\vf'\mapsto\jmath_{f}\cti\circ\vf'$
defines a map $\gamma_{f}\cti$ from $\Hfp\cti$ to $\Hf\cti$. Note
that the existence of a morphism $\vf'$ as above guaranties the fact
that both $\Hf\cti$ and $\Ff\cti$ are non-empty. When $\Hfp\cti=\emptyset$,
we let $\gf\cti$ to be the identity of the empty set. Obviously,
$\gf\cti$ is injective for all $\cti$ in the coslice category of
$\T$ under $C$. Furthermore, if $\varphi\in\Hf\cti$, then there
exists a $\iota$-factorization $\ti f$ such that $\varphi=\varphi_{\ti f}$.
Thus, the image of $\vf$ is contained in $\Ftf\cti$ and $\vf$ factors
through a morphism $\vf'\in\Hfp\cti$.  In conclusion, $\varphi=\jf\cti\circ\vf'=\gf\cti(\vf')$,
so $\gf\cti$ is bijective.
\end{proof}
\begin{claim}[The presheaf $\Hfp$]
\label{rem:F_f} Since $\gf\cti$ is invertible, for each morphism
$\xi:\cti\to(\widetilde{C}',\iota')$
in the coslice category of $\T$ under $C$, there exists a unique
function $\Hfp(\xi)$ that makes the following diagram commute:
\begin{equation}
\begin{array}{c}
\xymatrix{\Hf(\widetilde{C}',\iota')\ar[rr]^{\Hf(\xi)} &  & \Hf\cti\\
\Hfp(\widetilde{C}',\iota')\ar[u]^{\gamma_{f}(\widetilde{C}',\iota')}\ar@{-->}[rr]_{\Hfp(\xi)} &  & \Hfp\cti\ar[u]_{\gamma_{f}(\widetilde{C},\iota)}
}
\end{array}\label{eq:Fbar}
\end{equation}
It is not difficult to see that $\Hfp(\xi)(\vf')=\vf'_{\ti f}$, where
$\ti f:=\jf\cti\circ\vf'\circ\xi$. We leave to the reader the straightforward
task of verifying that the sets $\Hfp\cti$ and the maps $\Hfp(\xi)$
define a presheaf $\Hfp$, isomorphic with $\Hf$ via $\gf:=\{\gf\cti\}_{\cti\in\cosl C{\T}}$.

Taking into account Theorem \ref{prop:F=00003DH} and Theorem \ref{te:relative cofree},
we conclude that for every $\ti f$ in $\Ff\cti$, there exists a
unique $\vf'_{\ti f}\in\Hfp\cti$ such that $\ti f=\pf\cti\circ\vf_{\ti f}'$,
namely the corestriction of $\vf_{\ti f}$ onto $\Ftv\cti$. Thus,
$\jf\cti\circ\vf'_{\ti f}=\vf{}_{\ti f}$. On the other hand, every
$\vf'\in\Hfp\cti$ is equal to $\vf_{\ti f}'$, where $\ti f:=\pf\cti\circ\vf'$. 
\end{claim}

\begin{rem}
Let $\ti f$ be an $\iota$-factorization of $f$. Since $\vf_{\ti f}$
and $\vf_{\ti f}'$ share the same image, and all morphisms in $\Hf\cti$
and $\Hfp\cti$ are of this type, we get
\begin{equation}
\Ftf\cti=\sum_{\varphi\in\Hf\cti}\Im\varphi=\sum_{\vf'\in\Hfp\cti}\Im\vf'.\label{eq:Ttif}
\end{equation}
\end{rem}

\begin{rem}
\label{rem:T'}Let $\T'$ be another admissible class, such that $C$
and $\widetilde{C}$ are of type $\T'$ as well. Thus, $\Ftf\cti$
is a $\T\bigcap\T'$-coalgebra. Indeed, under the standing assumptions,
$\Im\varphi$ is a $\T'$-coalgebra, for all $\varphi$ in $\Hf\cti$. 
From the above equation, it follows that $\Ftf\cti$ is a $\T'$-coalgebra.
For instance, if $C$ and $\widetilde{C}$ are $\T$-coalgebras, and
are additionally pointed and cocommutative, then $\Ftf\cti$ either
inherits these properties or is zero.
\end{rem}

Working with connected coalgebras requires more attention since, for
this class of coalgebras, the condition $(\ddagger)$ does not hold
in general. Nevertheless, as an application of our previous results,
we obtain the following result.

\begin{thm}
\label{thm:connected_sigma-factorization} If $C$ and $\widetilde{C}$
are connected coalgebras such that $\Ff\cti\neq\emptyset$, then $\Ftf\cti$
is connected too. 
\end{thm}

\begin{proof}
Let $C_{0}$ denote the coradical of $C$. As a vector space $C_{0}$
has dimension one and is spanned by $x_{C}$, the unique group-like
element in $C$. Therefore, by \cite[Corollary 5.3.5]{Mo}, the coradical
of $\Im\sf$ is precisely the image of $C_{0}$ under $\sf$, which
is spanned by $\sf(x_{C})$. Hence, the image of $\sf$ is a connected
coalgebra, with the unique group-like element $\sf(x_{C})$. Proceeding
similarly, we can show that the image of $\vf$ is a connected coalgebra,
for any $\vf$ in $\Hf\cti$. The unique group-like element of $\Im\varphi$
is, of course, $\varphi(x_{C})$.  By the definition of $\Hf$, we
have $\vf\circ\iota=\sf$. Using the fact that coalgebra morphisms
send group-like elements to group-like elements, this equation implies
that $(\Im\varphi)_{0}$ is also spanned by $\sf(x_{C})$. Note that
this group-like element depends only on $f$. 

Let $\Ftf\cti_{0}$ denote the coradical of $\Ftf\cti$. By Equation
\eqref{eq:F_f(C,i)}, and in view of \cite[Proposition 8.0.3 (a)]{Sw},
we obtain
\[
\Ftf\cti_{0}=\sum_{\vf\in\Hf\cti}(\Im\varphi)_{0}.
\]
We deduce that $\Ftf\cti_{0}$ is spanned by $\sf(x_{C})$, meaning
that $\Ftf\cti$ is connected.
\end{proof}
\begin{example}
\label{ex:k}Let $\iota$ denote the canonical inclusion
$\Bbbk\hookrightarrow\Bbbk[t]$ of connected coalgebras. Any nonzero
linear map $f:\Bbbk\to V$ is uniquely determined by a nonzero vector
$v$, so $F_{v}^{\T}(\Bbbk[t],\iota):=\Ftf(\Bbbk[t],\iota)$ is a
cocommutative connected coalgebra. Here, $\T$ denotes the admissible
class of all cocommutative coalgebras. The problem of determining
this coalgebra is both intriguing and challenging.
\end{example}

\section{\texorpdfstring{$\iota$}{$\iota$}-Deformations of algebras in tensor categories\label{sec:Def}}

In this section we introduce the concept of $\iota$-deformation of
an algebra $(A,m)$ in $\M_{C}$. Here, $\iota$ denotes, as usual,
a coalgebra morphism from $C$ to $\ti C$, where both coalgebras
are cocommutative of some type $\T$. In our approach, deformations
are also structured as a presheaf on the category $\cosl C{\T}$.
In the main result of this section, Theorem \ref{thm:D=00003DC},
we prove an analog of Theorem \ref{te:relative cofree} for deformations.
\begin{claim}[Algebras in tensor categories]
 As explained in the introduction, in this paper we focus on the
study of algebras that are not necessarily unital. Henceforth, unless
otherwise stated, an algebra in an arbitrary monoidal category $(\mathfrak{M},\ot,\I)$
will denote an object $A$ together with an associative multiplication,
i.e., a morphism\emph{ }$m\in\mathfrak{M}(A^{\ot2},A)$ satisfying
the relation:
\begin{equation}
m\circ(m\ot A)=m\circ(A\ot m)\circ a_{A,A,A}.\label{eq:algebra}
\end{equation}
In light of the Coherence Theorem, we can omit the associativity constraint
from the right-hand side of the above equation.

A morphism between two algebras $(A,m)$ and $(A',m')$ is an element
$h\in\mathfrak{M}(A,A')$ such that
\[
h\circ m=m'\circ(h\ot h).
\]
For every algebra $(A,m)$ in $\mathfrak{M}$ and every invertible
morphism $h\in\mathfrak{M}(A,A)$, by transport of the algebra structure,
we get a new algebra $(A,m_{h})$, where 
\begin{equation}
m_{h}:=h^{-1}\circ m\circ(h\ot h).\label{eq:m_f}
\end{equation}
It is easy to see that $h:(A,m_{h})\to(A,m)$ is an isomorphism of
algebras in $\mathfrak{M}$. 

Algebras in $\mathfrak{M}$ and their morphisms form, of course, a
category. 

Returning to our favorite example of a monoidal category, we denote
by $\cat A_{C}$ the category of algebras in $\M_{C}$. If $\iota:C\to\widetilde{C}$
is a morphism of cocommutative coalgebras of type $\T$, then $\iota^{*}$
induces a functor from $\cat A_{\widetilde{C}}$ to $\cat A_{C}$,
still denoted by $\iota^{*}$. For example, $\varepsilon_{C}^{*}$
allows us to embed $\cat M$ and $\cat A$ into $\cat M_{C}$ and
$\cat A_{C}$, respectively.

To simplify the notation, the hom-space $\cat M(A^{\ot n},A)$ will
be denoted by $\mbb An$. By convention, $\mbb A0=\cat M(\I,A)$.
It is important to note that the multiplication of an algebra in $\M$
is a morphism in $\mbb A2$. In contrast, the multiplication of an
algebra in $\M_{C}$ is a $\Bbbk$-linear map from $C$ to $\mbb A2$. 
\end{claim}

\begin{example}[The algebras in $\cat M_{\Bbbk[t]}$.]
\label{ex: def_Gerst} By \S\ref{fa:k=00005Bt=00005D}, a morphism
$\widetilde{m}\in\M_{\Bbbk[t]}(A\ot A,A)$ can be identified to a
family $\{m_{n}\}_{n\in\N}$, where $m_{n}\in\mbb A2$ for all $n$.
It is easy the see that $(A,\widetilde{m})$ is an algebra in $\M_{\Bbbk[t]}$
if and only if, for all $n$, we have:
\[
\sum_{i+j=n}m_{i}\circ(m_{j}\ot A)=\sum_{i+j=n}m_{i}\circ(A\ot m_{j}).
\]
Clearly, $\cat A_{\Bbbk}$ and the category of algebras in $\M$ coincide,
and the functor $\iota^{*}:\cat A_{\Bbbk[t]}\to\cat{A_{\Bbbk}}$ maps
$(A,\{m_{n}\}_{n\in\N})$ to $(A,m_{0})$. Moreover, $(A,\widetilde{m})$
belongs to the fiber of $\iota^{*}$ over an algebra $(A,m)$ in $\cat A$
if and only if $m_{0}=m$.  In conclusion,  algebras in the fiber
of $\iota^{*}$ over $(A,m)$ are in one-to-one correspondence with
the classical deformations of $m$, as defined by Gerstenhaber in
\cite{G1}.
\end{example}

\begin{defn}
\label{de:Deformation} An $\iota$-\emph{deformation} of $(A,m)\in\cat A_{C}$
is an algebra $(A,{\it \widetilde{m}})\in\cat A_{\widetilde{C}}$
such that $\iota^{*}(A,\widetilde{m})=(A,m)$. By a \emph{deformation
of} $(A,m)$ we mean an $\iota$-deformation, for some coalgebra map
$\iota$ with domain $C$. 
\end{defn}

In view of the classic case, we might be tempted to define deformations
only for algebras in $\cat M=\cat M_{\Bbbk}$, that is to require
that $\iota\in\cat C(\Bbbk,\widetilde{C})$ or, equivalently, to pick
a group-like element in $\widetilde{C}$. However, for greater flexibility,
we prefer the above definition. A key advantage of this approach is
that it immediately leads us to the notion of deformation of a family
of algebras in $\cat M$, assuming that $C$ is pointed and cosemisimple.
Moreover, this more general definition will enable us, in the following
section, to address the classification of deformations using the theory
of coalgebra extensions.
\begin{defn}
\label{def:echiv_deformation}Two $\iota$-deformations $(A,\widetilde{m})$
and $(A,\widetilde{m}')$ of the same algebra $(A,m)$ in $\M_{C}$
are \emph{equivalent} if and only if there is an isomorphism $h\in\M_{\widetilde{C}}(A,A)$
such that $h\circ\iota=\varepsilon_{C}(-)I_{A}$
\[
\widetilde{m}'=\widetilde{m}_{h}\qquad\text{and}\qquad h\circ\iota=\varepsilon_{C}(-)I_{A}.
\]
Recall that $\widetilde{m}_{h}$ is defined as in Equation \eqref{eq:m_f}.
Since we now work in $\M_{\widetilde{C}}$, we have $\widetilde{m}_{h}:=h^{-1}*\widetilde{m}*(h\ot h)$. 
If $\widetilde{m}$ and $\widetilde{m}'$ are equivalent then we will
write $\widetilde{m}\sim\widetilde{m}'$.
\end{defn}

\begin{claim}[The presheaf $\Dm$ and the coalgebras $\Dtm\cti$ and $\Dtm$]
\label{claim:D^T_m} \label{claim:D_m} As with factorizations, we
associate to $(A,m)$ a presheaf $\Dm:\cosl C{\T}\to\cat S$ as follows:
\[
\Dm\cti:=\{(A,\widetilde{m})\in\cat A_{\widetilde{C}}\mid\widetilde{m}\circ\iota=m\}.
\]
On morphism $\Dm$ acts by $\Dm(\xi)(\widetilde{m}'):=\widetilde{m}'\circ\xi$,
where $\xi:\cti\to(\widetilde{C}',\iota')$ and $\widetilde{m}'\in\Dm(\widetilde{C}',\iota')$. 

Specializing our general results on factorizations of a linear map
to the case when $V:=\mbb A2$, $f:=m$, and $\ti f$ is equal to
the underlying factorization of a $\iota$-deformation $\widetilde{m}$
of $m$, we obtain the following commutative diagram:

\begin{equation}
\begin{array}{c}
\xymatrix{\widetilde{C}\ar[rrrr]^{\widetilde{m}}\ar[ddd]_{\vf{}_{\widetilde{m}}}\ar[rrdd]|{\vf'_{\widetilde{m}}} &  &  &  & \mbb A2\\
 &  & C\ar[rru]^{m}\ar[ull]_{\iota}\ar[rrdd]|(0.6){\iota_{m}}\\
 &  & \Ftm(\widetilde{C},\iota)\ar@{^{(}->}[dll]^{\jmath_{m}(\widetilde{C},\iota)}\ar@{_{(}->}[rrd]_{\jmath_{m}(\widetilde{C},\iota)}\ar[rruu]|(0.25){\hole}|{\ \ \pi_{m}(\widetilde{C},\iota)}\\
F_{\mbb A2}^{\T}\ar@{=}[rrrr] &  &  &  & F_{\mbb A2}^{\T}\ar[uuu]_{\pi_{\mbb A2}}
}
\end{array}\label{eq:diag_fac}
\end{equation}
Aiming to show that $\Dm$ is representable, we first introduce the
$\T$-coalgebra:
\begin{equation}
\Dtm\cti:=\sum_{\widetilde{m}\in\Dm\cti}\Im\varphi_{\widetilde{m}}.\label{eq:Dtm}
\end{equation}
Here it is understood that, if there are no $\iota$-deformations
of $m$, then $\Dtm\cti=0$. Clearly, $\Dtm\cti$ is a $\T$-subcoalgebra
of $\Ftm\cti$. We let:
\begin{equation}
\Dtm:=\sum_{\cti\in\cosl CT}\Dtm\cti.\label{eq:D_m}
\end{equation}
If $\widetilde{m}$ is an $\iota$-deformation of $m$, then $\Im\vf'_{\widetilde{m}}=\Im\vf_{\widetilde{m}}\subseteq\Dtm\cti\subseteq\Dtm$. 
Since $\iota_{m}=\jm\cti\circ\vf'_{\widetilde{m}}\circ\iota$,
the image of $\iota_{m}$ is contained in $\Dtm$. Henceforth, the
corestrictions of $\varphi_{\widetilde{m}}$ and $\vf'_{\widetilde{m}}$
onto $\Dtm$ and $\Dtm\cti$ will be denoted by $\psi_{\widetilde{m}}$
and $\psi'_{\widetilde{m}}$, respectively. Similarly, there is a
morphism of coalgebras $\iota'_{m}$, from $C$ to $\Dtm$, such that
$\iota_{m}'=\jm\circ\iota{}_{m}$, where $\jm$ is the canonical inclusion
$\Dtm\hookrightarrow F_{\mbb A2}^{\T}$. For the restriction of $\pi_{m}\cti$
to $\Dtm\cti$, we will retain the same notation, and $\pm$ will stand
for the restriction of $\pi_{\mbb A2}$ to $\Dtm$.  

To better understand the properties of these morphisms, we redraw
Diagram \eqref{eq:diag_fac} in the context of the current framework.

\begin{equation}
\begin{array}{c}
\xymatrix{\widetilde{C}\ar[rrrr]^{\tilde{m}}\ar[ddd]_{\psi_{\widetilde{m}}}\ar[rrdd]|{\psi'_{\tilde{m}}} &  &  &  & \mbb A2\\
 &  & C\ar[rru]^{m}\ar[ull]_{\iota}\ar[rrdd]|(0.6){\iota'_{m}}\\
 &  & \Dtm(\widetilde{C},\iota)\ar@{^{(}->}[dll]^{\jmath_{m}(\widetilde{C},\iota)}\ar@{_{(}->}[rrd]_{\jmath_{m}(\tilde{C},\iota)}\ar[rruu]|(0.25){\hole}|{\ \ \pi_{m}(\widetilde{C},\iota)}\\
\Dtm\ar@{=}[rrrr] &  &  &  & \Dtm\ar[uuu]_{\pi_{m}}
}
\end{array}\label{eq:diag_def}
\end{equation}
It is clear that this diagram commutes, as all arrows are derived
by appropriately restricting or corestricting the morphism from Diagram
\eqref{eq:diag_fac}. We are now in a position to construct the key
ingredient needed in the proof of the main result of this part, Theorem
\ref{thm:D=00003DH}.
\end{claim}

\begin{lem}
\label{lem:pi_m}The triple $(\Dtm,\sm',\pm)$ is an $\iota'_{m}$-deformation
of $m$.
\end{lem}

\begin{proof}
From the commutativity of \eqref{eq:diag_def}, it is clear that $(\Dtm,\sm',\pm)$
is an $\iota'_{m}$-factorization of $m$. To see that this triple
is an $\iota'_{m}$-deformation of $m$, we let $L(\nu):=\nu*(\nu\ot\Id A)$
and $R(\nu):=\nu*(\Id A{}\ot\nu)$, for every $\nu:D\to\mbb A2$,
where $D$ is an arbitrary cocommutative coalgebra.  Thus, we have
to prove that $L(\pm)=R(\pm)$. Taking into account the definition
of $\Dtm$ and the fact that $L(\pm)$ and $R(\pm)$ are linear maps
from $D_{m}^{\T}$ to $\mbb A3$, it is sufficient to show that their
restrictions to $\Im\varphi_{\widetilde{m}}$ are equal, for any $\cti$
in $\cosl CT$ and $\widetilde{m}\in\Dm\cti$. Let $z\in\widetilde{C}$. 
Since $\varphi_{\widetilde{m}}$ is a morphism of coalgebras, we have:
\begin{alignat*}{1}
L(\pm)(\varphi_{\widetilde{m}}(z)) & =\sum\pm(\varphi_{\widetilde{m}}(z)\r 1)\circ\big(\pm(\varphi_{\widetilde{m}}(z)\r 2)\ot I_{A}\big)\\
 & =\sum\pm(\varphi_{\widetilde{m}}(z\r 1))\circ\big(\pm(\varphi_{\widetilde{m}}(z\r 2))\ot I_{A}\big)\\
 & =\sum\widetilde{m}(z\r 1)\circ(\widetilde{m}(z\r 2)\ot I_{A}).
\end{alignat*}
We can compute $R(\pm)(\varphi_{\widetilde{m}}(z))$ in a similar
manner. Since $\widetilde{m}$ is associative, we conclude that $R(\pm)$
is equal to $L(\pm)$. 
\end{proof}
\begin{thm}
\label{thm:D=00003DH}The presheaves $\Dm$ and $\Dfr:=\cosl C{\T}(-,(D_{m},\iota'_{m}))$
are isomorphic.
\end{thm}

\begin{proof}
For any $\iota$-deformation $\widetilde{m}$ of $m$, the coalgebra
map $\psi_{\widetilde{m}}$ belongs to $\Dfr\cti$, as shown by a
simple computation. Thus, $\widetilde{m}\mapsto\psi_{\widetilde{m}}$
defines a map $\alm'\cti$ from $\Dm\cti$ to $\Dfr\cti$. We obtain
another function by mapping $\psi\in\Dfr\cti to \jm\circ\psi\in\mathfrak{F}_{m}\cti$.
This function is then used, along with the inclusion $\Dm\cti\subseteq\Fm\cti$
and the bijection $\alm\cti$, to construct the diagram:
\[
\begin{array}{c}
\xymatrix{\Dm\cti\ar@{^{(}->}[d]\ar[rr]^{\alm'\cti} &  & \Dfr\cti\ar[d]\\
\Fm\cti\ar[rr]_{\alm\cti} &  & \mathfrak{F}_{m}\cti
}
\end{array}
\]
Based on the identities $\alm\cti(\widetilde{m})=\vf_{\widetilde{m}}=\jm\circ\psi_{\widetilde{m}}$,
one can see that the diagram commutes. Obviously, $\alm'\cti$ is
injective, since $\alm\cti$ is so. 

To prove that it is surjective, we pick an arbitrary element $\psi$
in $\Dfr\cti$.  Note that the image of $\psi$ is included in $\Dtm\cti$. 
As $\alm\cti$ is surjective, there exists $\widetilde{m}\in\Fm\cti$
such that $\vf_{\widetilde{m}}=\alm\cti(\widetilde{m})=\jm\cti\circ\psi$.
Thus, $\Im\vf_{\widetilde{m}}=\Im\psi\subseteq\Dtm$. Since $\widetilde{m}=\pi_{\mbb A2}\circ\vf_{\widetilde{m}}$,
we can write $\widetilde{m}$ as a composition of an associative multiplication
(the restriction of $\pm$ to $\Im\vf_{\widetilde{m}}$, which is
associative in light of Lemma \ref{lem:pi_m}) and a coalgebra morphism
(the corestriction of $\vf_{\widetilde{m}}$ onto its image). Now,
knowing that $\widetilde{m}\in\Dm\cti$, we can check that $\psi_{\widetilde{m}}=\psi$,
so $\alm'\cti$ is surjective. 

Of course, the above discussion is meaningful only if there exist
$\iota$-deformations of $m$. But, if $\Dm\cti$ is empty, then $\Dfr\cti$
is also empty, so in this case $\alm'\cti$ is the identity of the
empty set.

If $\bem'\cti$ is the inverse of $\alm'\cti$, then $\bem'\cti(\psi)=\pm\circ\psi$,
provided that $\Dm\cti$ is non-empty. Otherwise, it is equal to the
identity of the empty set. As $\bem':=\{\bem'\cti\}_{\cti\in\cosl C{\T}}$
is a natural transformation, the presheaves $\Dfr$ and $\Dm$ are
isomorphic. 
\end{proof}
\begin{rem}
An $\iota$-factorization $\widetilde{m}$ of $m$ is associative,
if and only if the restriction of $\pi_{\mbb A2}$ to $\Im\vf_{\widetilde{m}}$
is associative. Indeed, the ``if'' part is a consequence of Lemma
\ref{lem:pi_m}, remarking that the restriction of $\pi_{\mbb A2}$
to $\Im\vf_{\widetilde{m}}$ and the restriction of $\pi_{m}$ to
the same set coincide. The ``only if'' part follows by the proof
of the theorem.
\end{rem}

\begin{claim}[The presheaf $\Dfr'$]
 As in the case of factorizations, there is a second presheaf $\Dfr'$
that may be used to study the deformations of an algebra $(A,m)$.
In principle, the construction of $\Dfr'$ and the proof that this
presheaf is isomorphic to $\Dm$, follow the same approach as in Subsection
\ref{rem:F_f}. Therefore, we will only outline the key steps, leaving
the details to the reader.

In the case when there are no $\iota$-deformations of $m$, we take
$\Dfr'\cti=\emptyset$. Otherwise, 
\[
\Dfr'\cti:=\big\{\psi'\in\cat C(\ti C, \Dtm)\mid\pm\cti\circ\psi'\circ\iota=m\big\}.
\]
For an arbitrary $\iota$-deformation $\widetilde{m}$ one shows,
using Diagram \eqref{eq:diag_fac}, that $\psi'_{\widetilde{m}}$ is
an element in $\Dfr'\cti$. Hence the existence of $\iota$-deformations
forces $\Dfr'\cti$ to be non-empty. Conversely, if $\psi'$ is a
coalgebra map in $\Dmp\cti$ then the relation 
\begin{equation}
\gm'\cti(\psi'):=\jm\cti\circ\psi'\label{eq:bem_bar}
\end{equation}
defines a function from $\Dfr'\cti$ to $\Dm\cti$.  The next step
is to demonstrate, following the reasoning in the proof of \ref{te:relative cofree},
that $\gm'\cti$ is bijective. Once this is established, for each
morphism $\xi$ from $\cti$ to $(\widetilde{C}',\iota')$, we define
$\Dfr'(\xi)$ as the unique function that ensures the commutativity
of the diagram below.
\begin{equation}
\begin{array}{c}
\xymatrix{\Dfr(\widetilde{C}',\iota')\ar[rr]^{\Dm(\xi)} &  & \Dfr\cti\\
\Dfr'(\widetilde{C}',\iota')\ar[u]^{\gm'(\widetilde{C}',\iota')}\ar@{-->}[rr]_{\Dfr'(\xi)} &  & \Dfr'\cti\ar[u]_{\gm'(\widetilde{C},\iota)}
}
\end{array}\label{eq:Dbar}
\end{equation}
In conclusion, the construction of the presheaf $\Dfr'$ is complete
and $\gm':=\{\gm'\cti\}_{\cti\in\cosl C{\T}}$ is an isomorphism between
$\Dm$ and $\Dfr'$.

Let us assume, in addition, that $C$ and $\widetilde{C}$ are pointed.
By Remark \ref{rem:T'}, the coalgebra $\Ftm\cti$ is either pointed
or zero. Thus, $\Dtm\cti$ is equal to zero or pointed, like any subcoalgebra
of $\Ftm\cti$.  Furthermore, $\Dtm$ is pointed as it is the sum of
these subcoalgebras, and $\Dtm(C,\Id C)$ is nonzero.

Similarly, if $C$ and $\widetilde{C}$ are connected and $\Dtm\cti\neq0$,
then the latter coalgebra is connected as well. Indeed, let $\cti$
be an object in the coslice category of $\T$ under $C$. By Theorem
\ref{thm:connected_sigma-factorization} it follows that $\Ftm\cti$
is connected. Hence, $\Dtm\cti$ is connected. By the proof of the
above mentioned theorem, we also know that the group-like element
of $\Dtm\cti$ is precisely $\sm(x_{C})$, which does not depend on
$\cti$. Thus, by \cite[Proposition 8.0.3 a)]{Sw} $\Dtm$ is a connected
coalgebra. All in all, we have proved the following theorem. 
\end{claim}

\begin{thm}
\label{thm:D=00003DC}Let $(A,m)$ denote an algebra in $\M_{C}$.
Then the presheaves $\Dm$ and $\Dfr'$ are isomorphic via $\gm'$.
If $C$ and $\widetilde{C}$ are pointed (respectively connected),
then $\Dtm$ and $\Dtm\cti$ are pointed (respectively connected). 
\end{thm}

\begin{example}
Let $\iota$ denote the canonical inclusion $\Bbbk\subseteq\Bbbk[t]$.
Any algebra in $\M_{\Bbbk}$ is uniquely determined by $m(1$), which
is an algebra in $\M$. The characterization of the connected cocommutative
coalgebra $\Dtm(\Bbbk[t],\iota)$ is also an interesting problem,
related to the question from Example \ref{ex:k}.
Additionally, it is intriguing to compare it with $\Ftm(\Bbbk[t],\iota)$.
Let us note that the letter coalgebra makes sense because $m$ is,
in particular, an $\iota$-factorization of $m$.
\end{example}

\section{The case of coalgebra extensions}

In this section, we will examine $\iota$-deformations of a given
algebra $(A,m)$ in $\M_{C}$, focusing on the case where $\iota:C\to\widetilde{C}$
is an extension of cocommutative coalgebras. We will see that this
is the appropriate framework for studying $\iota$-deformations using
homological methods, as in the classical case. The admissible coalgebras
are no longer of significant importance. It is sufficient that all
coalgebras we work with to be cocommutative. In other words, one may
now consider $\T$ as the admissible class of all cocommutative coalgebras.
\begin{claim}[Extensions of cocommutative coalgebras]
\label{claim:extension} By definition, an injective morphism of
coalgebras $\iota:C\to\widetilde{C}$ is an \emph{extension of coalgebras}
if 
\begin{equation}
\Delta_{\widetilde{C}}(\widetilde{C})\subseteq\widetilde{C}\ot\iota(C)+\iota(C)\ot\widetilde{C}.\label{eq:extension}
\end{equation}
Coalgebra extensions were defined and classified in \cite[Subsection 3.2]{Do}.
The key steps of the classification consists in the construction of
a $C$-bicomodule structure on $X:=\coker\iota$ and of a $2$-cocycle
of $C$ with coefficients in $X$.  This is carried out as follows.
First, one picks a linear retract $\lambda$ of $\iota$. We can assume
that $\lambda$ is normalized, that is $\varepsilon_{C}\circ\lambda=\varepsilon_{\widetilde{C}}$. 

Let $p:\widetilde{C}\to X$ denote the canonical projection. Then,
a right coaction $\rho_{r}$ of $C$ on $X$ is defined in a unique
way such that 
\begin{equation}
\rho_{r}\circ p=(p\ot\lambda)\circ\Delta_{\widetilde{C}}.\label{eq:rho}
\end{equation}
 A left coaction $\rho_{l}$ is defined by a similar relation. In
order to characterize the comultiplication of $\ti C$, one shows
that the unique linear map $\omega:X\to C\ot C$ satisfying the equation
\begin{equation}
\omega\circ p=(\lambda\ot\lambda)\circ\Delta_{\widetilde{C}}-\Delta_{C}\circ\lambda\label{eq:omega}
\end{equation}
 is a $2$-cocyle with coefficients in $X$, i.e., 
\begin{equation}
(C\ot\omega)\circ\rho_{l}-(\Delta_{C}\ot C)\circ\omega+(C\ot\Delta_{C})\circ\omega-(\omega\ot C)\circ\rho_{r}=0.\label{eq:2-cocycle}
\end{equation}
We use the Sweedler notation:

\[
\rho_{l}(x)=\sum x\u{-1}\ot x\u 0\qquad\rho_{r}(x)=\sum x\u 0\ot x\u 1\qquad\omega(x)=\sum\omega_{1}(x)\ot\omega_{2}(x).
\]
Let $\iota_{X}$ denote the canonical inclusion $X\subseteq C\oplus X$.
Via the identification $\widetilde{C}\simeq C\oplus X$ (as vector
spaces), which maps $z\in\widetilde{C}$ to $(\lambda(z),p(z))$,
the comultiplication $\Delta_{\widetilde{C}}$ satisfies the relation:
\begin{equation}
\begin{aligned}\Delta_{\ti C}(c,x):=\sum(c\r 1,0)\ot(c\r 2,0) & +\sum(x\u{-1},0)\ot(0,x\u 0)+\sum(0,x\u 0)\ot(x\u 1,0)+\\
 & +\sum(\omega_{1}(x),0)\ot(\omega_{2}(x),0)
\end{aligned}
\label{eq:Delta_tild}
\end{equation}
The normalization condition we imposed on $\lambda$ is necessary
for $\omega$ to be a \emph{normalized} cocycle, i.e., 
\begin{equation}
\sum\varepsilon_{C}(\omega_{1}(x))\omega_{2}(x)=0=\sum\omega_{1}(x)\varepsilon_{C}(\omega_{2}(x)).\label{eq:omega-normalizat}
\end{equation}
In turn, $\omega$ needs to be normalized because only in this case
does $\varepsilon_{\widetilde{C}}$ satisfy the relation $\varepsilon_{\widetilde{C}}(c,x):=\varepsilon_{C}(c)$.

If $\widetilde{C}$ is cocommutative, it is easy to see that $\rho_{l}(x)=\sum x\u 1\ot x\u 0$,
i.e., $X$ is a \emph{symmetric} bicomodule. Also, for cocommutative
coalgebras, the corresponding $2$-cocycle $\omega$ is \emph{symmetric},
so 
\begin{equation}
\sum\omega_{1}(x)\ot\omega_{2}(x)=\sum\omega_{2}(x)\ot\omega_{1}(x).\label{eq:omega_simetric}
\end{equation}
Conversely, given a coalgebra $C$, a $C$-bicomodule $(X,\rho_{l},\rho_{r})$
and a normalized $2$-cocycle $\omega$ with coefficients in $X$,
the direct sum $\widetilde{C}=C\oplus X$ becomes a coalgebra with
respect to the comultiplication \eqref{eq:Delta_tild}. The counit
maps $(c,x)$ to $\varepsilon_{C}(c)$.  Of course, $\iota$ is an
extension of $C$ with cokernel $X$ and the canonical projection $\lambda$
onto $X$ is a normalized retract of $\iota$.  With respect to this
normalized retract, the corresponding bicomodule and $2$-cocycle
are precisely $(X,\rho_{l},\rho_{r})$ and $\omega$, respectively.

Obviously, if $C$ is cocommutative and $X$ and $\omega$ are symmetric,
then $\ti C$ is cocommutative as well. 
\end{claim}

\begin{claim}[Classification of deformations for cocommutative coalgebra extensions]
\label{claim:.Def_coalg_ext} We now fix an extension of cocommutative
coalgebras $\iota:C\to C\oplus X$, corresponding to a symmetric $C$-bicomodule
$X$ and a symmetric normalized $2$-cocycle $\omega$.  We are going
to describe the set $\Dm\cti$ in this setting. Let $\widetilde{m}$
be an $\iota$-factorization of $m$; which is regarded simply as
a linear map. Then
\[
\widetilde{m}(c,x)=\widetilde{m}(\iota(c)+\iota_{X}(x))=\widetilde{m}(c,0)+\widetilde{m}(0,x).
\]
By definition of $\iota$-factorizations, $\widetilde{m}\circ\iota=m$. 
Thus,
\[
\widetilde{m}(c,x)=m(c)+(\widetilde{m}\circ\iota_{X})(x)=m(c)+\widetilde{m}_{X}(x),
\]
where $\widetilde{m}_{X}:=\widetilde{m}\circ\iota_{X}$. Taking into
account Equation \eqref{eq:Delta_tild} and the fact that $\sum x\u{-1}\ot x\u 0=\sum x\u 1\ot x\u 0$,
we obtain the following two identities:
\begin{alignat*}{1}
[\widetilde{m}*(\widetilde{m}\ot A)](c,x)=\sum & m(c\r 1)\circ\big(m(c\r 2)\otimes A\big)+\sum\widetilde{m}_{X}(x\u 0)\circ\big(m(x\u 1)\otimes A\big)+\\
 & +\sum m(x\u 1)\circ\big(\widetilde{m}_{X}(x\u 0)\ot A\big)+\sum m(\omega_{1}(x))\circ\big(m(\omega_{2}(x))\ot A\big),\\{}
[\widetilde{m}*(A\ot\widetilde{m})](c,x)=\sum & m(c\r 1)\circ\big(A\ot m(c\r 2)\big)+\sum\widetilde{m}_{X}(x\u 0)\circ\big(A\ot m(x\u 1)\big)+\\
 & +\sum m(x\u 1)\circ\big(A\ot\widetilde{m}_{X}(x\u 0)\big)+\sum m(\omega_{1}(x))\circ\big(A\ot m(\omega_{2}(x))\big).
\end{alignat*}
For an arbitrary element $\nu\in\cat V(X,\mbb A2)$, we define $d^{2}(\nu)\in\cat V(X,\mbb A3)$
by
\begin{alignat}{1}
d^{2}(\nu)(x):=\sum & m(x\u 1)\circ\big(A\ot\nu(x\u 0)\big)-\sum\nu(x\u 0)\circ\big(m(x\u 1)\otimes A\big)+\label{eq:d^2}\\
 & +\sum\nu(x\u 0)\circ\big(A\ot m(x\u 1)\big)-\sum m(x\r 1)\circ\big(\nu(x\u 0)\ot A\big).\nonumber 
\end{alignat}
 Let $\zeta:X\to\mbb A3$ denote the $\Bbbk$-linear map given by:
\end{claim}
\begin{equation}
\zeta(x):=\sum m(\omega_{1}(x))\circ\big(A\ot m(\omega_{2}(x))\big)-\sum m(\omega_{1}(x))\circ\big(m(\omega_{2}(x))\ot A).\label{eq:zeta}
\end{equation}
Since $m$ is associative, it results that $\widetilde{m}$ is a deformation
of $m$ if and only if $\widetilde{m}_{X}$ is a solution of the \emph{generalized
Maurer-Cartan} \emph{equation, }that is, we have: 
\begin{equation}
d^{2}(\widetilde{m}_{X})+\zeta=0.\label{eq:d^(m)=00003Dxi}
\end{equation}
Note that, in particular, the generalized Maurer-Cartan equation has
a solution if and only if $\Dm\cti\neq\emptyset$. Assuming that there
exists an $\iota$-deformation $\widetilde{m}$ of $m$, then the
set of all solutions of \eqref{eq:d^(m)=00003Dxi} coincides with
$\widetilde{m}_{X}+\mrm[Z][2]$, where $\mrm[Z][2]:=\Ker d^{2}$.
Equivalently, the mapping $\widetilde{m}'\mapsto\widetilde{m}_{X}'-\widetilde{m}_{X}$
defines a bijective function from $\Dm\cti$ to $\mrm[Z][2]$. The
inverse sends $\nu\in\mrm[Z][2]$ to $\widetilde{m}_{\nu}$, the deformation
defined by $\widetilde{m}_{\nu}(c,x)=\widetilde{m}(c,x)+\nu(x)$. 
Thus, we have established the following result.
\begin{thm}
\label{thm:Def=Sol_MCE} The generalized Maurer-Cartan equation
has a solution if and only if $\Dm\cti\neq\emptyset$. If $\iota$-deformations
of $m$ exist they correspond one-to-one with the elements of the
set $\mrm[Z][2]$. 
\end{thm}

\begin{claim}[Classification of equivalent deformations in the case of cocommutative
coalgebra extensions]
\label{claim:equiv_def_extensions} We keep the assumption that $\widetilde{C}$
is an extension of $C$, as above. Let $f\in\M_{\widetilde{C}}(A,A)$
be an arbitrary morphism. We can apply the results from Subsection \ref{fa:iTakeuchi Lemma},
for $D'=C$ and $D''=\widetilde{C}$, to prove that $f$ is invertible
if and only if its restriction to $C$ is invertible. In particular,
any $f$ satisfying $\iota^{*}(f)=\Id A{}$ is invertible in $\M_{\widetilde{C}}$. 

Let $\widetilde{m}'$ and $\widetilde{m}''$ be two  $\iota$-deformations
of $m$. We denote the corresponding solutions of the generalized
Maurer-Cartan equation by $\widetilde{m}_{X}'$ and $\widetilde{m}''_{X}$,
respectively. Let $f$ be a morphism in $\M_{\widetilde{C}}(A,A)$
such that $f(c,0)=\varepsilon_{C}(c)I_{A}$. Thus,
there exists a unique linear map $f_{X}:X\to\mbb A1$ satisfying the
relation $f(c,x)=\varepsilon_{C}(c)I_{A}+f_{X}(x)$. Hence, we have:
\begin{alignat*}{1}
(f*\widetilde{m}'')(c,x)= & \sum f(c\r 1,0)\circ\widetilde{m}''(c\r 2,0)+\sum f(x\u 1,0)\circ\widetilde{m}''(0,x\u 0)+\\
 & \qquad+\sum f(0,x\u 0)\circ\widetilde{m}''(x\u 1,0)+\sum f(\omega_{1}(x),0)\circ\widetilde{m}''(\omega_{2}(x),0)\\
= & m(c)+\widetilde{m}''_{X}(x)+\sum f_{X}(x\u 0)\circ m(x\u 1)+m\big(\sum\varepsilon_{C}(\omega_{1}(x))\omega_{2}(x)\big).\\
= & m(c)+\widetilde{m}''_{X}(x)+\sum f_{X}(x\u 0)\circ m(x\u 1).
\end{alignat*}
 Note that the last equation holds since $\omega$ is normalized.
On the other hand,
\begin{alignat*}{1}
(f\ot f)(c,x) & =\sum f(c\r 1,0)\ot f(c\r 2,0)+\sum f(x\u 1,0)\ot f(0,x\u 0)+\sum f(0,x\u 0)\ot f(x\u 1,0)+\\
 & \qquad+\sum f(\omega_{1}(x),0)\ot f(\omega_{1}(x),0)\\
 & =\varepsilon_{C}(c)I_{A}\ot I_{A}+I_{A}\ot f_{X}(x)+f_{X}(x)\ot I_{A}+\big[\sum\varepsilon_{C}(\omega_{1}(x))\varepsilon_{C}(\omega_{2}(x))\big]I_{A\ot A}\\
 & =\varepsilon_{C}(c)I_{A\ot A}+I_{A}\ot f_{X}(x)+f_{X}(x)\ot I_{A}.
\end{alignat*}
Therefore, by a straightforward computation, we get 
\begin{alignat*}{1}
\big(\widetilde{m}'*(f\ot f)\big)(c,x)=m(c) & +\widetilde{m}'_{X}(x)+\sum m(x\u 1)\circ\big(f_{X}(x\u 0)\ot A\big)+\\
 & +\sum m(x\u 1)\circ\big(A\ot f_{X}(x\u 0)\big).
\end{alignat*}
In conclusion, $\widetilde{m}'$ and $\widetilde{m}''$ are equivalent
if and only if there exists $f_{X}:X\to\mbb A1$ such that
\begin{equation}
\widetilde{m}_{X}''=\widetilde{m}'_{X}+d^{1}(f_{X}).\label{eq:d^1(f_X)}
\end{equation}
Here, for any $\nu$ in $\cat V(X,\mbb A1)$, the map $d^{1}(\nu)\in\cat V(X,\mbb A2)$
is defined by the relation:
\[
d^{1}(\nu)(x)=\sum m(x\u 1)\circ\big(A\ot\nu(x\u 0)\big)-\sum\nu(x\u 0)\circ m(x\u 1)+\sum m(x\u 1)\circ\big(\nu(x\u 0)\ot A\big).
\]
We denote the image of $d^{1}$ by $\mrm[B][2]$. We thus obtain the
following result.
\end{claim}

\begin{thm}
\label{thm:Equiv_Def} Two deformations $\widetilde{m}'$ and $\widetilde{m}''$
of $(A,m)$ are equivalent if and only if the corresponding solutions
of the generalized Maurer-Cartan equation satisfies the relation $\widetilde{m}'_{X}-\widetilde{m}_{X}''\in\mrm[B][2]$. 
\end{thm}

\begin{claim}[{The cohomology $\mrm[H][*]$}]
\label{claim:C*}In this subsection, for a given algebra in $(A,m)$
in $\M_{C}$ and an extension of cocommutative coalgebras $C\subseteq\widetilde{C}$,
we relate the elements of $\mrm[Z][2]$ and $\mrm[B][2]$ to the $2$-cocycles
and the $2$-coboundaries in a certain cochain complex $\mrm[C][*]$.
Specifically, we show that the construction of $\mrm[C][*]$ can be
carried out for any right $C$-comodule, not just for $X=\coker\iota$. 

We begin by defining a semi-cosimplicial object $\mrm[C][*]$ in $\cat V$,
with coface maps $d_{i}^{n}:\mrm[C][*]\to\mrm[C][n+1]$, where $0\leq i\leq n+1$.
Recall that, by definition, the coface morphisms must satisfy the
commutation relations $d_{j}^{n+1}\circ d_{i}^{n}=d_{i}^{n+1}\circ d_{j-1}^{n}$,
for all $i<j\leq n+2$. 

We define $\mrm[C][n]$ as the vector space of linear functions from
$X$ to $\mbb An$. Recall that, by convention, $\mbb A0=\cat M(\I,A)$,
so we can set
\[
d_{i}^{0}(\nu)(x):=\begin{cases}
m(x\u 1)\circ(A\ot\nu(x\u 0)\circ r_{A}^{-1}, & \text{if }i=0;\\
m(x\u 1)\circ(\nu(x\u 0\ot A)\circ l_{A}^{-1}, & \text{if }i=1.
\end{cases}
\]
On the other hand, for $n>0$, we let:
\begin{alignat}{1}
d_{i}^{n}(\nu)(x) & :=\begin{cases}
\sum m(x\u 1)\circ\big(A\ot\nu(x\u 0\big), & \text{if }i=0;\\
\sum\nu(x\u 0)\circ\big(A^{\ot i-1}\ot m(x\u 1)\ot A^{\ot n-i}\big), & \text{if }1\leq i\leq n;\\
\sum m(x\u 1)\circ\big(\nu(x\u 0)\ot A\big), & \text{if }i=n+1.
\end{cases}\label{eq:coface}
\end{alignat}
  The commutation relations follow from standard, though somewhat
tedious, computations. As an example, we consider the case $i=0$
and $j=1$, leaving the proof of the remaining cases to the reader.
If $n>0$,  we have:

\begin{align*}
\big(d_{0}^{n+1}\circ d_{0}^{n}\big)(\nu)(x) & =\sum m(x\u 1)\circ\big(A\ot d_{0}^{n}(\nu)(x\u 0)\big)\\
 & =\sum m(x\u 2)\circ\big((A\ot m(x\u 1)\big)\circ\big(A^{\ot2}\ot\nu(x\u 0)\big)\\
 & =\sum\big(m*(A\ot m)\big)(x\u 1)\circ\big(A^{\ot2}\ot\nu(x\u 0)\big).
\end{align*}
To deduce the last equation, we used the fact that $C$ is cocommutative
and $X$ is a right $C$-comodule, so
\[
\sum(x\u 0)\u 0\ot(x\u 0)\u 1\ot x\u 1=\sum x\u 0\ot x\u 1\ot x\u 2=\sum x\u 0\ot(x\u 1)\r 1\ot(x\u 1)\r 2.
\]
On the other hand, 
\begin{align*}
\big(d_{1}^{n+1}\circ d_{0}^{n}\big)(\nu)(x) & =\sum d_{0}^{n}(\nu)(x\u 0)\circ\big(m(x\u 1)\ot A^{\ot n}\big)\\
 & =\sum m(x\u 1)\circ\big(A\ot\nu(x\u 0)\big)\circ\big(m(x\u 2)\ot A^{\ot n}\big)\\
 & =\sum m(x\u 1)\circ\big(m(x\u 2)\ot A\big)\circ\big(A^{\ot2}\ot\nu(x\u 0)\big)\\
 & =\sum\big(m*(m\ot A)\big)(x\u 1)\circ\big(A^{\ot2}\ot\nu(x\u 0)\big).
\end{align*}
Thus, $d_{1}^{n+1}\circ d_{0}^{n}$ and $d_{0}^{n+1}\circ d_{0}^{n}$
coincide, since $m$ is associative. 
\end{claim}

\begin{claim}
To our semi-cosimplicial object corresponds a complex with cochain
groups $\mrm[C][*]$ and differentials $d_{X}^{*}:=\sum_{i=0}^{n+1}(-1)^{i}d_{i}^{*}$.
The vector spaces of $n$-cocycles and $n$-coboundaries are denoted
by $\mrm[Z][n]$ and $\mrm[B][n]$, respectively. This notation is
consistent with the one we used in Subsections \ref{claim:.Def_coalg_ext} and
\ref{claim:equiv_def_extensions}, since $d_{X}^{1}$ and $d_{X}^{2}$
are equal to $d^{1}$ and $d^{2}$. In particular, we have $d^{2}\circ d^{1}=0$.
We denote the cohomology of $\mrm[C][*]$ by $\mrm[H][*]$.
\end{claim}

\begin{rem}
\label{rem:standard_complex}Let $C:=\Bbbk$ and let $X$ be the trivial
$C$-comodule. If $m^{0}\in\mbb A2$ is an algebra in $\M$, we can
consider the complex $\mrm[C][*]$, where $m:C\to\mbb A2$ is the
unique multiplication in $\M_{C}$ such that $m(1)=m^{0}$. To avoid
confusion, we will denote this complex by $\mathrm{C^{*}}(A,m^{0})$.
Furthermore, in this particular case, we will use the notation $\partial_{i}^{n}:=d_{i}^{n}$. 

The cohomology of $\mathrm{C^{*}}(A,m^{0})$ is precisely the Hochschild
cohomology of $(A,m^{0})$ with coefficients in $A$, regarded as
a bimodule over itself. For the definition of Hochschild cohomology
of algebras in monoidal categories see, for instance, \cite[§1.24]{AMS}.

We will later need a few elementary properties of the complex $\mrm[C][*]$.
These are stated and proved in the following lemma.
\end{rem}

\begin{lem}
\label{lem:prop_C^*_X} Let $C^{*}$ denote the dual algebra of $C$. 
\begin{enumerate}
\item[\emph{(a)}]  The complex $\mrm[C][*]$ is a cochain complex in the category of
right $C^{*}$-modules. 
\item[\emph{(b)}]  If $X$ is the direct sum of a family $\{X_{i}\}_{i\in I}$ then
$\mrm[C][*]\simeq\prod_{i\in I}\mathrm{C}_{X_{i}}^{*}(A,m)$.
\end{enumerate}
\end{lem}

\begin{proof}
Any right $C$-comodule is a left $C^{*}$-module. Thus $\cat V(X,V)$
is a right $C^{*}$-module with respect to the action
\[
(\nu\leftharpoonup\alpha)(x)=\nu(\alpha\rightharpoonup x),
\]
for any linear map $\nu\in\cat V(X,V)$ and $\alpha\in C^{*}$. Recall
that $\alpha\rightharpoonup x=\sum\alpha(x\u 1)x\u 0$. Therefore,
$\mrm[C][n]$ is a right $R$-module and it is not difficult to see
that: 
\begin{alignat*}{1}
d_{i}^{n}(\nu\leftharpoonup\alpha)(x) & :=\begin{cases}
\sum\alpha(x\u 1)m(x\u 2)\circ\big(A\ot\nu(x\u 0\big), & \text{if }i=0;\\
\sum\nu(x\u 0)\circ\big(A^{\ot i-1}\ot\alpha(x\u 1)m(x\u 2)\ot A^{\ot n-i}\big), & \text{if }1\leq i\leq n;\\
\sum\alpha(x\u 1)m(x\u 2)\circ\big(\nu(x\u 0)\ot A\big), & \text{if }i=n+1.
\end{cases}
\end{alignat*}
Expanding $\big(d_{i}^{n}(\nu)\leftharpoonup\alpha\big)(x)$, we obtain
the same result, proving that the coface maps $d_{i}^{n}$ are $C^{*}$-linear.
Therefore, $\mrm[C][*]$ is a cochain complex in the category of right
$C^{*}$-modules. 

For the second part of the lemma, we observe that, for every $n$,
the map $\nu\mapsto\{\nu\res{X_{i}}{}\}_{i\in I}$ defines an isomorphism
of $C^{*}$-modules between $\mrm[C][n]$ and $\prod\mathrm{C}_{X_{i}}^{n}(A,m)$.
A straightforward computation shows that $d_{i}^{n}(\nu)\res{X_{i}}{}=d_{i}^{n}(\nu\res{X_{i}}{})$,
which implies that the aforementioned isomorphisms are compatible
with the differential maps.
\end{proof}
\begin{thm}
\label{thm:main_result_for_extensions} Let $\Dm\cti/_{\sim}$ denote
the set of equivalence classes of $\iota$-deformations of $(A,m)$.
\begin{enumerate}
\item[\emph{(a)}] The cochain $\zeta$ is always a $3$-cocycle.
\item[\emph{(b)}] The cohomology class $\zeta$  is trivial if and only if $\Dm\cti\neq\emptyset$.
\item[\emph{(c)}] The bijective map from Theorem \ref{thm:Def=Sol_MCE} induces
a one-to-one correspondence between $\Dm\cti/_{\sim}$ and $\mrm[H][2]$.
\end{enumerate}
\end{thm}

\begin{proof}
In light of Theorem \ref{thm:Def=Sol_MCE} and Theorem \ref{thm:Equiv_Def},
it remains to prove that $\zeta\in\mrm[Z][3]$. Throughout the proof, we use the simplifying notation $m\lambda:=m\circ\lambda$ and $\zeta p:=\zeta\circ p$.
For any $z\in\widetilde{C}$, using Equation \eqref{eq:omega} and
the definition of $\zeta$, we obtain:
\begin{alignat*}{1}
\zeta p(z) & =\sum m \lambda(z\r 1)\circ\big(A\ot m\lambda(z\r 2)\big)-\sum m\lambda(z\r 1)\circ\big(m\lambda(z\r 2)\ot A\big)\\
 & =[m\lambda*(A\ot m\lambda)-m\lambda*(m\lambda\ot A)](z).
\end{alignat*}
On the other hand, for $x=p(z)$,  Equation \eqref{eq:rho} can
be written as follows
\[
\sum p(z)\u 0\ot p(z)\u 1=\sum p(z\r 1)\ot\lambda(z\r 2)=\sum p(z\r 2)\ot\lambda(z\r 1).
\]
Now we can compute the value of the coface morphisms $d_{k}^{3}$
at $\zeta$. For $k=0$ we have: 
\begin{align*}
d_{0}^{3}(\zeta)(p(z)) & =\sum m(p(z)\u 1)\circ\big((A\ot\zeta(p(z)\u 0)\big)\\
 & =\sum m\lambda(z\r 1)\circ\big(A\ot\zeta p(z\r 2)\big)\\
 & =[m\lambda*\big(A\ot\zeta p\big)](z)\\
 & =\big(m\lambda*(A\ot m\lambda)*(A\ot A\ot m\lambda)-m\lambda*(A\ot m\lambda)*(A\ot m\lambda\ot A)\big)(z).
\end{align*}
Proceeding similarly, for $k\in\{1,2,3\}$, we get:
\[
d_{k}^{3}(\zeta)p=m\lambda*(A\ot m\lambda)*(A^{\ot k-1}\ot m\lambda\ot A^{\ot3-k})-m\lambda*(m\lambda\ot A)*(A^{\ot k-1}\ot m\lambda\ot A^{\ot3-k}).
\]
The last coface operator is given by 
\[
d_{4}^{3}(\zeta)p=m\lambda*(m\lambda\ot A)*(A\ot m\lambda\ot A)-m\lambda*(m\lambda\ot A)*(m\lambda\ot A\ot A).
\]
In view of the above computations, we conclude that $d^{3}(\zeta)p=\sum_{i=0}^{4}(-1)^{i}d_{i}^{3}(\zeta)p=0$.
Since $p$ is surjective, it follows that $\zeta$ is a $3$-cocycle.
\end{proof}
\begin{rem}
\label{rem:triv_ext}An extension of cocommutative coalgebras $\iota:C\to\widetilde{C}$
is said to be \emph{trivial} if $\iota$ has a left inverse $\lambda$
that is a morphism of coalgebras. In this case, the corresponding
$2$-cocycle $\omega:X\to C\ot C$, defined as in \S\ref{claim:extension},
is equal to zero. Hence the \emph{obstruction to deformation} (i.e.,
the $3$-cocycle $\zeta$) is also equal to zero. Consequently, by
Theorem \ref{thm:main_result_for_extensions}, any algebra $(A,m)$
admits an $\iota$-deformation. In fact, we can directly observe that
$\Dm\cti$ is non-empty because $m\circ\lambda$ is a deformation
of $m$. Moreover, by the proof of Theorem \ref{thm:Def=Sol_MCE},
it follows that for any $\iota$-deformation $\widetilde{m}$, there
exists a 2-cocycle $\nu$ such that
\[
\ti m(c,x)=(m\circ\lambda)(c,x)+\nu(x).
\]
The class of trivial extensions includes, for instance, all $\iota:C\to\ti C$,
where $C$ is coseparable. For the definition of coseparable coalgebras
and the fact that these extensions are trivial, see \cite{Do}. 
\end{rem}

\section{Particular cases }

In this section we will specialize some of our results to several
classes of extensions $\iota:C\to\ti C$ arising from a graded cocommutative
coalgebra. Additionally, in the graded case, we will prove that an
$\iota$-deformation of an associative algebra with unit is equivalent
to a unital deformation.
\begin{claim}
\label{claim:graded_coalg}Let $D=\oplus_{n\in\N}D^{n}$ denote a
graded cocommutative coalgebra. If $n$ is a nonnegative integer,
then the Equation \eqref{eq:extension} holds for $C=D_{<n}=:\oplus_{k<n}D^{k}$
and $\ti C=D_{\leq n}:=\oplus_{k\leq n}D^{k}$. Consequently, the
inclusion $\iota:C\to\ti C$ defines an extension of cocommutative
coalgebras. Adopting the notation from \ref{claim:extension}, we
have $X=D^{n}$ and $p$ coincides with the canonical projection of
$D_{\leq n}$ onto $D^{n}$. We choose the projection map $\lambda:\ti C\to C$
as linear retract of $\iota$. Hence, the right $C$-coaction on $X$
is given, for every $x\in X$, by:
\begin{equation}
\rho_{r}(x):=(p\ot\lambda)\big(\sum_{i=0}^{n}\sum x\r{1,i}\ot x\r{2,n-i}\big)=\sum x\r{1,n}\ot x\r{2,0}=\sum x\r{2,n}\ot x\r{1,0}.\label{eq:rho_gr}
\end{equation}
Note that the second identity follows from cocommutativity of $D$. 
Since $\rho_{r}(x)\in X\ot D^{0}$, we can view $X$ as a comodule
over both $C$ and $D^{0}$.

Furthermore, the $2$-cocycle $\omega$ turns out to be equal to:
\begin{equation}
\omega(x)=\sum_{i=1}^{n-1}\Delta^{i,n-i}(x)=\sum_{i=1}^{n-1}\sum x\r{1,i}\ot x\r{2,n-i}.\label{eq:omega_gr}
\end{equation}
Therefore, if $(A,m)$ is an associative multiplication in $\M_{C}$,
then the obstruction to deformation is given by the formula:
\begin{equation}
\zeta(x)=\sum_{i=1}^{n-1}\sum m(x\r{1,i})\circ\big(A\ot m(x\r{2,n-i})-m(x\r{2,n-i})\ot A\big).\label{eq:zeta_gr}
\end{equation}
For $n>0$, a direct application of Equation \eqref{eq:rho_gr} yields:
\begin{alignat*}{1}
d_{i}^{n}(\nu)(x) & :=\begin{cases}
\sum m(x\r{1,0})\circ\big(A\ot\nu(x\r{2,1}\big), & \text{if }i=0;\\
\sum\nu(x\r{1,0})\circ\big(A^{\ot i-1}\ot m(x\r{2,1})\ot A^{\ot n-i}\big), & \text{if }1\leq i\leq n;\\
\sum m(x\r{1,0})\circ\big(\nu(x\r{2,1})\ot A\big), & \text{if }i=n+1.
\end{cases}
\end{alignat*}
Having $\rho_{r}$, $\zeta$ and the differentials of the complex
$\mrm[C][*]$ determined,  Theorem \ref{thm:main_result_for_extensions}
can be applied in the current context. The output is identical to
the previously mentioned result, so we will omit it here and instead
focus on exploring several more specific cases.
\end{claim}

\begin{claim}[Infinitesimal deformations]
\label{rem:infinitesimal} We take $n=1$ in the previous subsection,
so $C:=D_{0}\subseteq D_{\leq1}=:\ti C$. Since $\lambda$, the canonical
projection of $\ti C$ onto $C$ is a coalgebra map, this extension
is trivial. Thus, any algebra $(A,m)$ in $\M_{C}$ admits a deformation.
Moreover, all of them are like those from Remark \ref{rem:triv_ext}.
In this context an $\iota$-deformation will be referred to as \emph{infinitesimal}.
The terminology is inspired by that used by Gerstenhaber. We will
return to the study of infinitesimal deformations when we discuss
the case where $D$ is graded connected. 
\end{claim}

\begin{claim}[Multiplications of rank one]
 At this level of generality, no further conclusions can be drawn
regarding $\mrm[H][*]$ and, implicitly, the set $\Dm\cti/_{\sim}$.
To gain additional insights, we will examine the more specialized
case where the rank of the multiplication $m$ is $1$. 

By definition, the \emph{rank} of $(A,m)$ is defined as being the
dimension of $m(D^{0})$. Thus, $m$ has rank 1 if and only if there
exists a map $\chi$ in the dual linear space of $D^{0}$ and an algebra
$(A,m^{0})$ in $\M$, such that $m(c)=\chi(c)m^{0}$ for all $c\in D^{0}$. 
We treat $\chi$ as an element in $C^{*}$, extending it by zero on
the components $D^{i}$, for any $0<i<n$. 

Recall that the maps $\partial_{i}^{n}$ are the coface morphism of
$\mathrm{C}^{*}(A,m^{0})$, see Remark \ref{rem:standard_complex}.
Given an algebra $(A,m)$ of rank 1, we can express the coface morphism
$d_{0}^{n}$ of $\mathrm{C}_{X}^{*}(A,m)$ as follows:
\[
\begin{alignedat}{1}d_{0}^{n}(\nu)(x)= & \sum m(x\r{1,0})\circ\big(A\ot\nu(x\r{2,n})\big)=m^{0}\circ\big[A\ot\nu\big(\sum\chi(x\r{1,0})x\r{2,n}\big)\big]\\
= & \partial{}_{0}^{n}\big(\nu(\chi\rightharpoonup x)\big)=\partial{}_{0}^{n}\big((\nu\leftharpoonup\chi)(x)\big)=\big(\partial{}_{0}^{n}\circ(\nu\leftharpoonup\chi)\big)(x).
\end{alignedat}
\]
Proceeding similarly, we show that $d_{i}^{n}=\partial{}_{i}^{n}\circ(\nu\leftharpoonup\chi)$,
to finally obtain that:
\begin{equation}
d_{X}^{n}(\nu)=\partial{}^{n}\circ(\nu\leftharpoonup\chi)=\cat V(X,\partial^{n})(\nu\leftharpoonup\chi).\label{eq: rank_1}
\end{equation}
More can be said if we further assume that $\chi=\varepsilon_{D^{0}}$.
Since $\nu\leftharpoonup\varepsilon_{D^{0}}=\nu$,  by the above equation,
we have $d_{X}^{*}=\cat V(X,\partial^{*})$. Thus, under the standing
assumptions, the complex $\mrm[C][*]$ is obtained from $\mathrm{C}^{*}(A,m^{0})$
by applying the exact functor $\cat V(X,-)$. Hence, $\mrm[H][*]$
is equal to $\cat V(X,\mathrm{H}^{*}(A,m^{0}))$. Then, according
to Theorem \ref{thm:main_result_for_extensions}, we have the following
result.
\end{claim}

\begin{thm}
\label{thm:rank_1}Let $\iota$ denote the inclusion map $C:=D_{<n}\subseteq D_{\leq n}=:\ti C$,
where $D=\oplus_{n\geq0}D^{n}$ is a graded coalgebra. Let $(A,m)$
be an algebra in $\M_{C}$ of rank 1, for which $\chi=\varepsilon_{D^{0}}$.
If the obstruction to deformation from Equation \eqref{eq:zeta_gr}
is a coboundary, then there is a one-to-one correspondence between
$\Dm\cti/_{\sim}$ and $\cat V(X,\mathrm{H}^{2}(A,m^{0}))$.
\end{thm}

\begin{claim}[The comodule $D^{n}$ is completely reducible]
\label{claim:completly_reducible} Infinitesimal deformations can
be seen as a first approximation of a deformation. The case we are
currently studying allows us, at least theoretically, to continue
the deformation process to a higher degree. In practice, the situation
is more complicated because, in the first place, the extension $C\subseteq\ti C$
is not necessarily trivial, so the obstruction to deformation may
be non-trivial. However, we can provide information about the structure
of the cohomology $\mrm[H][*]$, at least for those extensions of
coalgebras where the comodule $X$ is completely reducible.

By definition, $X$ is called \emph{completely reducible} if and only
if $X=\oplus_{i\in I}X_{i}$, where each $X_{i}$ is a one-dimensional
subcomodule over $C$. We fix such a decomposition of $X$. By Equation
\eqref{eq:rho_gr}, the image of $\rho_{r}$ is contained in $D^{n}\ot D^{0}$,
so there are $\{x_{i}\}_{i\in I}\subseteq X$ and $\{g_{i}\}_{i\in I}\subseteq G(D^{0})$
such that $\rho_{r}(x_{i})=x_{i}\ot g_{i}$, for all $i\in I$. 

If $(A,m)$ is an algebra in $\M_{C}$, then we let $m_{i}:=m(g_{i})$,
for all $i\in I$. Of course, every $m_{i}\in\mbb A2$ is associative.
Taking into account Lemma \ref{lem:prop_C^*_X}, we get the identifications
\[
\mrm[C][*]\simeq\prod_{i\in I}\mathrm{C}_{X_{i}}^{*}(A,m)\simeq\prod_{i\in I}\mathrm{C}^{*}(A,m_{i}).
\]
As we already know, the first isomorphism maps a cochain $\nu\in\mrm[C][n]$
to $\{\nu\res{X_{i}}{}\}_{i\in I}$. On the other hand, the second
isomorphism is given by the mapping $\{\nu_{i}\}_{i\in I}\mapsto\{\nu_{i}(x_{i})\}_{i\in I}$.
Since cohomology and direct products commute, we deduce that 
\[
\mrm[H][*]\simeq\prod_{i\in I}\mathrm{H}^{*}(A,m_{i}).
\]
Summarizing, we have the following theorem.
\end{claim}

\begin{thm}
\label{thm:completely_reducible}Let $\iota$ denote the inclusion
map $C:=D_{<n}\subseteq D_{\leq n}=:\ti C$. Let $D$ be a graded
coalgebra such that $D^{n}$ is completely reducible. If $(A,m)$
is an algebra in $\M_{C}$ such that the corresponding obstruction
to deformation $\zeta$ is a coboundary, then there is a one-to-one
correspondence between $\Dm\cti/_{\sim}$ and $\prod_{i\in I}\mathrm{H}^{2}(A,m_{i})$.
\end{thm}

Examples of graded coalgebras that satisfy the conditions of the above
theorem include pointed coalgebras, for which the degree-zero component
is cosemisimple. If $D$ is such a coalgebra, then any $D^{0}$-comodule
can be written as a direct sum of simple subcomodules. On the other
hand, $D^{0}$ being pointed, all simple $D^{0}$-comodule are of
dimension $1$. Thus, any $D^{0}$-comodule is completely reducible.
Using the notation from Subsection \ref{claim:completly_reducible}
and previous theorem, we are now in a position to state the corollary
below. 
\begin{cor}
\label{cor:D=00003Dpointed_cosemi}We assume that $D$ is a pointed
graded coalgebra, such that $D^{0}$ is cosemisimple. If the corresponding
obstruction to deformation $\zeta$ is a coboundary, then there is
a one-to-one correspondence between $\Dm\cti/_{\sim}$ and $\prod_{i\in I}\mathrm{H}^{2}(A,m_{i})$.
\end{cor}

Let $D:=\Bbbk[t_{1},\dots,t_{r}]$ be the vector space of polynomials
in variables $t_{1},\dots,t_{r}$ with coefficients in $\Bbbk$.  If
$P:=(p_{1},\dots,p_{r})\in\N^{r}$, then $t^{P}:=t_{1}^{p_{1}}\cdots t_{r}^{p_{r}}$. 
The degree of $t^{P}$ will be denoted by $\left|P\right|$. We endow
$D$ with a cocommutative coalgebra structure by letting: 
\[
\Delta(t^{P}):=\sum_{Q+R=P}t^{Q}\ot t^{R}.
\]
Let $D^{n}$ denote the subspace spanned by the set of all monomials
of degree $n$. By definition, $\varepsilon_{D}$ is zero on $D^{n}$,
for all $n>0$, and coincides with the identity of $\Bbbk$ on $D^{0}$. 
The subspaces $\{D^{n}\}_{n\in\N}$ provide a grading of $D$. Note
that the dimension of $D^{n}$ is equal to $d_{n}:=\binom{n+r-1}{n}$.

An immediate application of Corollary \ref{cor:D=00003Dpointed_cosemi},
for $D:=\Bbbk[t_{1},\dots,t_{r}]$, yields us the following corollary.
\begin{cor}
\label{cor:poly_r}Let $(A,m)$ be an algebra in $\M_{C}$. If the
corresponding obstruction to deformation $\zeta$ is a coboundary,
then there is a one-to-one correspondence between $\Dm\cti/_{\sim}$
and $\mathrm{H}^{2}(A,m(1))^{d_{n}}$. 
\end{cor}

\begin{rem}
The classification of deformations of an algebra $(A,m)$ can be derived
directly from the preceding corollary by setting $r=1$. By considering
Remark \ref{rem:standard_complex}, it follows that the cohomology
$\mathrm{H^{*}}(A,m(1))$ corresponds exactly to the Hochschild cohomology
of the algebra $(A,m(1))$ within $\M$,  as defined in \cite{AMS}.
By specializing the category $\M$, various specific deformation theories
emerge, including those of $H$-actions and coactions, cochain algebras
($DG$-algebras), diagrams of associative algebras, and others.
\end{rem}

In the last part of the article we address the problem of the existence
of a unit in the context of deformations of unital algebras. More
precisely, let $D=\oplus_{n\in\N}D^{n}$ be a graded coalgebra. Let
$\lambda:\ti C\to C$ and $\iota:C\to\ti C$ denote the canonical
maps, where $C:=D^{0}$ and $\ti C:=D$. Our aim is to prove the following
result.
\begin{thm}
\label{thm:unit} If $(A,m,u)$ is a unital algebra in $\M_{C}$ and
$(A,\ti m)$ is an $\iota$-deformation of $(A,m)$, then there exists
an invertible morphism $f\in\M_{\ti C}(A,A)$ such that $u\circ\lambda$
is a unit of $\ti m_{f}:=f^{-1}\ast\ti m\ast(f\ot f)$. Furthermore,
$f\ast(u\circ\lambda)$ is a unit of $\ti m$. 
\end{thm}

\begin{proof}
To simplify notation, we will denote the composition of two morphisms
$f$ and $g$ in $\cat V$ by $fg$, rather than $f\circ g$. Since
$D$ is graded, $\lambda$ is a coalgebra retract of $\iota$. Let
$\phi:=\iota\lambda$. We aim to construct a sequence $\{f_{n}\}_{n\in\N}$
of invertible morphisms in $\M_{\ti C}(A,A)$, such that $f_{n}\equiv_{n-1}f_{n-1}$,
for any $n>0$, and $\ti m_{f_{n}}:=f_{n}^{-1}\ast\ti m\ast(f_{n}\ot f_{n})$
satisfies the relations: 
\begin{equation}
\ti m_{f_{n}}*(u\lambda\ot\Id A{})*l_{A}^{-1}\equiv_{n}\Id A{}\equiv_{n}\ti m_{f_{n}}*(\Id A{}\ot u\lambda)*r_{A}^{-1}.\tag{$U_{n}$}\label{eq:unit_n}
\end{equation}
Assuming this sequence has been constructed, we proceed to conclude
the proof of the theorem. Since $f_{p}$ and $f_{p+1}$ are $p$-congruent
for all $p$, it follows that $f_{n}(c)=f_{p}(c)$,  provided that
$c\in D_{n}$ and $p\geq n$.  Thus, we can define $f:\ti C\to\mbb A1$
such that $f(z):=f_{n}(z)$, for any $z\in D_{\leq n}$. Obviously,
$f\equiv_{n}f_{n}$ for all $n$, and $f$ is invertible, because
$f\equiv_{0}\mathrm{Id}_{A}$. Hence, for any $n$, the maps $\tilde{m}_{f}$
and $\ti m_{f_{n}}$ are $n$-congruent, and
\[
\ti m_{f}*(u\lambda\ot\mathrm{Id}_{A})*l_{A}^{-1}\equiv_{n}\ti m_{f_{n}}*(u\lambda\ot\mathrm{Id}_{A})*l_{A}^{-1}\equiv_{n}\mathrm{Id}_{A}.
\]
Similarly, $\ti m_{f}*(\mathrm{Id}_{A}\ot u\lambda)*r_{A}^{-1}\equiv_{n}\mathrm{Id}_{A}$,
for all $n$. Thus, $u\lambda$ is a unit for $\ti m_{f}$. Then,
$f\ast u\lambda$ is also a unit of $\ti m$. 

It remains to prove the existence of $\{f_{n}\}_{n\in\N}$. We proceed
by induction. For $n=0$, we let $f_{0}:=\mathrm{Id}_{A}$. Hence
$\ti m_{f_{0}}=\ti m$ and the relation ($U_{0}$) holds in this case
since $u$ is a unit of $m$. Let us assume that we have already defined
$f_{0},\dots,f_{n}$ satisfying the required properties. We are searching
for $f_{n+1}$ of the following form: 
\[
f_{n+1}:=f_{n}\ast(\Id A{}+g_{n+1})=f_{n}+f_{n}*g_{n+1},
\]
for some $g_{n+1}\in\cat V(\ti C,\mbb A2)$. For now, we only require
that $g_{n+1}\equiv_{n}0$. It follows that the map $\Theta$, defined
by $\Theta(c'\ot c'')=f_{n}(c')g_{n+1}(c'')$, vanishes on $D_{\leq n}\ot D_{\leq n}$.
Note that $g_{n+1}\phi=0$. Also, $f_{n}\phi=\Id A{}$ since $f_{n}$
and $f_{0}$ coincide on $C$ and $\varepsilon_{\ti C}=\varepsilon_{C}\lambda$.
By Equation \eqref{eq:Theta'}, it follows 
\[
f_{n}\ast g_{n+1}\equiv_{n+1}(f_{n}\phi)*g_{n+1}+f_{n}*(g_{n+1}\phi)=g_{n+1}.
\]
Clearly, $f_{n+1}\equiv_{n}f_{n}+g_{n+1}\equiv_{n}f_{n}$. Note that
$f_{n+1}$ is an invertible morphism, being the composition in $\M_{\ti C}$
of two invertible maps; see Takeuchi Lemma. By definition of $f_{n+1}$,
we get: 
\[
m_{f_{n+1}}=(\Id A{}+g_{n+1})^{-1}\ast m_{f_{n}}\ast\big((\Id A{}+g_{n+1})\ot(\Id A{}+g_{n+1})\big).
\]
By Equation \eqref{eq:u_inverse}, the inverse of $\Id A{}+g_{n+1}$
is $(n+1$)-congruent to $\mathrm{Id}_{A}-g_{n+1}$. On the other
hand, 
\[
(\mathrm{Id}_{A}+g_{n+1})\ot(\mathrm{Id}_{A}+g_{n+1})\equiv_{n+1}\id_{A\ot A}+\mathrm{Id}_{A}\ot g_{n+1}+g_{n+1}\ot\mathrm{Id}_{A},
\]
since $g_{n+1}\ot g_{n+1}\equiv_{n+1}0$. It follows: 
\begin{alignat*}{1}
\ti m_{f_{n+1}}\equiv_{n+1}\ti m_{f_{n}} & +\ti m_{f_{n}}*(g_{n+1}\ot\mathrm{Id}_{A}+\mathrm{Id}_{A}\ot g_{n+1})-g_{n+1}*\ti m_{f_{n}}\\
 & -g_{n+1}*\ti m_{f_{n}}*(g_{n+1}\ot\mathrm{Id}_{A}+\mathrm{Id}_{A}\ot g_{n+1}).
\end{alignat*}
To simplify the above expression of $\ti m_{f_{n+1}}$, we will apply
Lemma \ref{le:computation_f} twice.

Let $\Theta_{1}:\ti C\ot\ti C\rightarrow\mbb A1$ denote the linear
map: 
\[
\Theta_{1}(x\ot y)=\ti m_{f_{n}}(x)\big(g_{n+1}(y)\ot A+A\ot g_{n+1}(y)\big)-g_{n+1}(x)\ti m_{f_{n}}(y).
\]
Since $g_{n+1}\equiv_{n}0$, it follows that $\Theta_{1}$ is zero
on $D_{\leq n}^{\ot2}$. Similarly, if $\Theta_{2}:\ti C^{\ot3}\rightarrow\mbb A1$
is given by: 
\[
\Theta_{2}(x\ot y\ot z)=g_{n+1}(x)\ti m_{f_{n}}(y)[g_{n+1}(z)\ot A+A\ot g_{n+1}(z)],
\]
then $\Theta_{2}$ vanishes on $\ti C_{n}^{\ot3}$. We know that $g_{n+1}\phi=0$
and $\ti m_{f_{n}}\phi=m\lambda$. By Lemma \ref{le:computation_f}
we have
\[
\Theta_{1}\Delta\equiv_{n+1}m\lambda*(g_{n+1}\ot\mathrm{Id}_{A}+\mathrm{Id}_{A}\ot g_{n+1})-g_{n+1}*m\lambda\qquad\text{and}\qquad\Theta_{2}\Delta^{2}\equiv_{n+1}0.
\]
In conclusion, 
\[
\ti m_{f_{n+1}}\equiv_{n+1}\ti m_{f_{n}}+m\lambda*(g_{n+1}\ot\Id A{}+\Id A{}\ot g_{n+1})-g_{n+1}*m\lambda.
\]
Furthermore, since $u$ is a unit of $m$, we get: 
\[
\ti m_{f_{n+1}}*(u\lambda\ot\mathrm{Id}_{A})\equiv_{n+1}\ti m_{f_{n}}*(u\lambda\ot\Id A{})+m\lambda*(g_{n+1}\ot\Id A{})*(u\lambda\ot\Id A{}).
\]
Let $\theta{}_{n}$ denote the right-hand side of the above congruence
relation. By induction hypothesis, the first term of $\theta{}_{n}$
is $n$-congruent to $l_{A}$. The second term is $n$-congruent to
0, as $g_{n+1}\equiv_{n}0$.  It follows that $\theta{}_{n}\equiv_{n}l_{A}$. 
Then, $(U_{n+1})$ is true if and only if $\theta{}_{n}\equiv^{n+1}l_{A}$,
that is $\theta{}_{n}-l_{A}$ is zero on $D^{n+1}$. Summarizing,
the first congruence relation of ($U_{n+1}$) is equivalent to 
\begin{equation}
\tag{$U'_{n+1}$}\ti m_{f_{n}}*(u\lambda\ot\Id A{})*l_{A}^{-1}+m\lambda*(g_{n+1}\ot\Id A{})*(u\lambda\ot\Id A{})*l_{A}^{-1}\equiv^{n+1}0.\label{unit_n'}
\end{equation}
Proceeding in a similar way, one shows that $\ti m_{f_{n+1}}$ satisfies
the second relation in $(U_{n+1})$ if and only if 
\begin{equation}
\tag{$U''_{n+1}$}\ti m_{f_{n}}*(\mathrm{Id}_{A}\ot u\lambda)*r_{A}^{-1}+m\lambda*(\mathrm{Id}_{A}\ot g_{n+1})*(\Id A{}\ot u\lambda)*r_{A}^{-1}\equiv^{n+1}0.\label{unit_n''}
\end{equation}
In conclusion, to construct $f_{n+1}$ we must find a morphism $g_{n+1}\in\M_{\ti C}(A,A)$
which vanishes on $D_{\leq n}$ and verifies the congruence relations
$(U'_{n+1})$ and $(U''_{n+1})$. 

By the universal property of the coproduct, there is a unique linear
map $g_{n+1}$ such that $g_{n+1}\equiv^{i}0$ for $i\neq n+1$, and
\[
g_{n+1}\equiv^{n+1}-\ti m_{f_{n}}*(\Id A{}\ot u\lambda)*r_{A}^{-1}.
\]
 In order to prove $(U_{n+1}')$, we denote the second term from
the left-hand side of this relation by $-\zeta_{n}$. Thus,
\begin{alignat*}{1}
-\zeta_{n} & \equiv^{n+1}m\lambda*(\ti m_{f_{n}}\ot\Id A{})*(\Id A{}\ot u\lambda\ot\Id A{})*(r_{A}^{-1}\ot\Id A{})*(u\lambda\ot\Id A{})*l_{A}^{-1}\\
 & \equiv^{n+1}m\lambda*(\Id A{}\ot m_{f_{n}})*(\Id A{}\ot u\lambda\ot\Id A{})*(r_{A}^{-1}\ot\Id A{})*(u\lambda\ot\Id A{})*l_{A}^{-1}\\
 & \stackrel{(1)}{=}m\lambda*(u\lambda\ot\Id A{})*(\Id A{}\ot\ti m_{f_{n}})*(\Id A{}\ot u\lambda\ot\Id A{})*(r_{\I}^{-1}\ot\Id A{})*l_{A}^{-1}\\
 & \stackrel{(2)}{=}l_{A}*(\Id A{}\ot\ti m_{f_{n}})*(\Id A{}\ot u\lambda\ot\Id A{})*(r_{\I}^{-1}\ot\Id A{})*l_{A}^{-1}\\
 & \stackrel{(3)}{=}l_{A}*(\Id{\I}{}\ot\ti m_{f_{n}})*(\Id{\mathbf{\I}}{}\ot u\lambda\ot\Id A{})*(\Id{\I}{}\ot l_{A}^{-1})*l_{A}^{-1}\\
 & \stackrel{(4)}{=}l_{A}*l_{A}^{-1}*\ti m_{f_{n}}*(u\lambda\ot\Id A{})*l_{A}^{-1}=\ti m_{f_{n}}*(u\lambda\ot\Id A{})*l_{A}^{-1}.
\end{alignat*}
For the congruence relation we used Equation \eqref{eq:Theta'}, setting
\[
\Theta(x\ot y)=\ti m_{f_{n}}(x)*\big[(\ti m_{f_{n}}\ot\Id A{}-\Id A{}\ot\ti m_{f_{n}})*(\Id A{}\ot u\lambda\ot\Id A{})*(r_{A}^{-1}\ot\Id A{})*(u\lambda\ot\Id A{})*l_{A}^{-1}\big](y).
\]
Note that $\Theta$ is trivial on $D_{\leq n}\ot D_{\leq n}$, since
$\ti m_{f_{n}}$ is associative and $(U_n)$ holds by induction
hypothesis. Furthermore, (1) and (4) follow by the fact that $r_{A}$
and $l_{A}$ are natural linear maps, (2) holds since $u\lambda$
is a unit of $m\lambda$ and the Triangle Axiom implies (3). 

The proof of $(U_{n+1}'')$  is analogous. We now define $\Theta$
by 
\[
\Theta(x\ot y):=\ti m_{f_{n}}(x)\big[(\ti m_{f_{n}}\ot\Id A{}-\Id A{}\ot\ti m_{f_{n}})*(\Id A{}\ot\Id A{}\ot u\lambda)*(\Id A{}\ot r_{A}^{-1}\big)*(\Id A{}\ot u\lambda)*r_{A}^{-1}\big](y).
\]
Since $\Theta$ is zero on $D_{\leq n}\ot D_{\leq n}$, we can apply
Equation \eqref{eq:Theta'} to obtain that $\zeta_{n}'\equiv^{n+1}\zeta_{n}''$,
where
\begin{alignat*}{1}
\zeta_{n}' & =m\lambda*(\ti m_{f_{n}}\ot\Id A{})*(\Id A{}\ot\Id A{}\ot u\lambda)*(\Id A{}\ot r_{A}^{-1})*(\Id A{}\ot u\lambda\big)*r_{A}^{-1},\\
\zeta_{n}'' & =m\lambda*\big(\Id A{}\ot\ti m_{f_{n}})*(\Id A{}\ot\Id A{}\ot u\lambda)*(\Id A{}\ot r_{A}^{-1})*(\Id A{}\ot u\lambda\big)*r_{A}^{-1}.
\end{alignat*}
The first term of the sum in $(U_{n+1}'')$ and $\zeta'_n$
are identical, since $u\lambda$ is a unit of $m\lambda$ and $r_{A}^{-1}$
is natural. By definition, $\zeta''_{n}=-g_{n+1}$,  so $(U_{n+1}'')$
holds. In conclusion, the theorem is completely proved.
\end{proof}

\subsection*{Final remark}

The approach to deformation theory developed here for associative
algebras can be extended to accommodate Lie algebras and bialgebras
as well. However, these extensions require additional structural assumptions
about the category $\M$. Specifically, since bialgebras are only
well-defined in braided monoidal categories, the category $\M$ must
possess a braided monoidal structure. Similarly, for Lie algebras,
$\M$ must be symmetric (e.g. the category of super vector spaces,
which leads to the study of deformations of super Lie algebras).

\section*{Acknowledgments}

The second-named author was financially supported by CNCS-UEFISCDI,
project CNCSIS PN-III-P4\_PCE-2021-0282, Contract 47/2022. Part of
this work was done during a visit to the University of Mulhouse. He
would like to thank the Laboratoire de Mathématiques, Informatique
et Applications (LMIA) for its worm hospitality.


\begin{thebibliography}{LVdB}

\bibitem[AMS]{AMS}A. Ardizzoni, C. Menini and D. \c{S}tefan, \emph{Hochschild
cohomology and 'smoothness' in monoidal categories}, J. Pure Appl.
Algebra 208 (2007), no. 1, 297--330.


\bibitem[BM]{BM}M. Bordemann and A. Makhlouf, \emph{Formality and
deformations of universal enveloping algebras}, Internat. J. Theoret.
Phys. 47 (2008), no.2, 311--332.



\bibitem[DL1]{DL1} H. Dinh Van and W. Lowen, \emph{The Gerstenhaber-Schack
complex for prestacks} Adv. Math. 330 (2018), 173--228.

\bibitem[DHL]{DHL}H. Dinh Van, Hoang, L. Hermans and W. Lowen, \emph{Operadic
structure on the Gerstenhaber-Schack complex for prestacks} Selecta
Math. (N.S.) 28 (2022), no.3, Paper No. 47, 63 pp.

\bibitem[Doi]{Do} Y. Doi, \emph{Homological coalgebra}, J. Math.
Soc. Japan 33 (1981), no. 1, 31-50.


\bibitem[ELS]{ELS}E. Eriksen, O. A. Laudal and A. Siqveland, \emph{Noncommutative
deformation theory,} Monogr. Res. Notes Math. CRC Press, Boca Raton,
FL, 2017. xvi+242 pp.


\bibitem[Fi]{Fi} A. Fialowski, \emph{Deformations of Lie algebras},
Mat. Sbornyik USSR, 127 (169), (1985), 476--482; English translation:
Math. USSR-Sb., \textbf{55}, no. 2 , (1986), 467--473.

\bibitem[FF]{FF}A. Fialowski and D. Fuchs, \emph{Construction of
versal deformation of Lie algebra}, Journal of Functional Analysis
\textbf{161} (1999), 76--110.


\bibitem[FN]{FN} A. Fröhlicher and A. Nijenhuis, \emph{A theorem
on stability of complex structures}, Proc. Nat. Acad. Sci. U.S.A.
43 (1957), 239--241.

\bibitem[G1]{G1} M. Gerstenhaber, \emph{On the deformation of rings
and algebras}, Ann. of Math. 79 (1964), 59--103.

\bibitem[G2]{G2} M. Gerstenhaber, \emph{On the deformation of rings
and algebras III}, Ann. of Math. 84 (1966), 1--19.

\bibitem[G3]{G3} M. Gerstenhaber, \emph{On the deformation of rings
and algebras IV}, Ann. of Math. 88 (1968), 1--34.

\bibitem[G4]{G4} M. Gerstenhaber, \emph{On the deformation of rings
and algebras I}, Ann. of Math. 99 (1974), 257--276. 

\bibitem[GS1]{GS1}M. Gerstenhaber and S.D. Schack, \emph{On the deformation
of algebra morphisms and diagrams}, Trans. Amer. Math. Soc. 279 (1)
(1983) 1--50. 

\bibitem[GS2]{GS2} M. Gerstenhaber and S.D. Schack, \emph{Algebraic
cohomology and deformation theory}, in: Deformation Theory of Algebras
and Structures and Applications, Il Ciocco, 1986, in: NATO Adv. Sci.
Inst. Ser. C Math. Phys. Sci., vol. 247, Kluwer Acad. Publ., Dordrecht,
1988, pp. 11--264. 

\bibitem[GS3]{GS3}M. Gerstenhaber and S.D. Schack, \emph{The cohomology
of presheaves of algebras. I. Presheaves over a partially ordered
set, }Trans. Amer. Math. Soc. 310 (1) (1988) 135--165.

\bibitem[Ha]{Haz} M. Hazewinkel, \emph{Cofree coalgebras and multivariable
recursiveness}, J. Pure Appl. Algebra 183 (2003), 61--103.

\bibitem[Ka]{Ka} C. Kassel, \emph{Quantum groups}, Graduate Text
in Mathematics, \textbf{155}, Springer-Verlag, New York, 1995.



\bibitem[LVdB]{LVdB}W. Lowen and M. Van den Bergh, \emph{ Deformation
theory of abelian categories}, Trans. Amer. Math. Soc. 358 (2006),
no.12, 5441--5483.

\bibitem[Ma]{Maj2} S. Majid, \emph{Foundations of quantum group theory},
Cambridge University Press, 1995.

\bibitem[McL]{McL} S. Mac Lane, \emph{Categories for the working mathematician}.
Second edition. Graduate Texts in Mathematics, Springer-Verlag, New
York, 1998.


no.2, 241--254.


\bibitem[Mo]{Mo} S. Montgomery, \emph{Hopf Algebras and Their Actions
on Rings}, CBMS Regional Conference Series in Mathematics, 82. Published
for the Conference Board of the Mathematical Sciences, Washington,
DC; by the American Mathematical Society, Providence, RI, 1993.

\bibitem[NR1]{NR1}A. Nijenhuis and Jr., R. W. Richardson, \emph{Cohomology
and deformations in graded Lie algebras}, Bull. Amer. Math. Soc.,
72, (1966), 1--29.

\bibitem[NR2]{NR2}A. Nijenhuis and Jr., R. W. Richardson, \emph{Deformations
of Lie algebra structures}, J. Math. Mech., 17, (1967), 89--105.

\bibitem[Sc]{Sc} M. Schlessinger, \emph{Functors of Artin rings,}
Transactions of the American Mathematical Society, vol. 130, no. 2,
pp. 208--222, 1968.

\bibitem[Sw]{Sw}M. Sweedler, \emph{Hopf algebras}, W. A. Benjamin, 1969.


\bibitem[We]{We}C. Weibel,\emph{ An introduction to homological algebra},
Cambridge University Press, Cambridge, 1994. 
\end{thebibliography}
\end{document}